\documentclass[a4paper,12pt]{article}
\usepackage[english]{babel}
\usepackage[utf8]{inputenc}
\usepackage{geometry}
\usepackage{hyperref}
\usepackage{amssymb}
\usepackage{amsmath}
\usepackage{amsthm}
\usepackage{makecell}
\usepackage{blkarray}
\usepackage{comment}
\usepackage{tikz}
\usetikzlibrary{cd}
\usetikzlibrary{backgrounds}
\allowdisplaybreaks
\title{Isotropic and Coisotropic Subvarieties \\ of Grassmannians}
\author{Kathlén Kohn, James Mathews}

\theoremstyle{definition}
\newtheorem{defn}{Definition}
\newtheorem{example}[defn]{Example}
\newtheorem{defCor}[defn]{Definition/Corollary}	
\newtheorem{remDef}[defn]{Remark/Definition}	

\theoremstyle{plain}
\newtheorem{thm}[defn]{Theorem}
\newtheorem{lem}[defn]{Lemma}
\newtheorem{prop}[defn]{Proposition}

\newtheorem{cor}[defn]{Corollary}
\newtheorem{open}{Question}

\theoremstyle{remark}
\newtheorem{rem}[defn]{Remark}

\newcommand{\bs}[1]{\boldsymbol{#1}}

\newcommand{\pl}{\mathrm{pl}}
\newcommand{\TT}{\mathbb{T}}
\newcommand{\ZZ}{\mathbb{Z}}
\newcommand{\CC}{\mathbb{C}}
\newcommand{\PP}{\mathbb{P}}
\newcommand{\KK}{\mathbb{K}}

\newcommand{\Hom}{\mathrm{Hom}}

\newcommand{\Reg}{\mathrm{Reg}}
\newcommand{\Sing}{\mathrm{Sing}}
\newcommand{\Seg}{\mathrm{Seg}}
\newcommand{\Sec}{\mathrm{Sec}}
\newcommand{\Osc}{\mathrm{Osc}}
\newcommand{\rank}{\mathrm{rank}}
\newcommand{\spann}{\mathrm{span}}

\newcommand{\Gr}{\mathrm{Gr}}
\newcommand{\im}{\mathrm{im}\,}
\newcommand{\codim}{\mathrm{codim}\,}
\newcommand{\isom}{\cong}
\newcommand{\longdashrightarrow}{\,\dashrightarrow}

\usepackage{scalerel}
\newlength\bshft
\def\fakebold#1{\ThisStyle{\ooalign{$\SavedStyle#1$\cr%
  \kern-\bshft$\SavedStyle#1$\cr%
  \kern\bshft$\SavedStyle#1$}}}
\newcommand{\aff}[1]{\fakebold{\fakebold{#1}}}

\begin{document}
\maketitle
\abstract{We generalize the notion of coisotropic hypersurfaces to subvarieties of Grassmannians having arbitrary codimension. 
To every projective variety $X$, Gel'fand, Kapranov and Zelevinsky associate a series of coisotropic hypersurfaces in different Grassmannians.
These include the Chow form and the Hurwitz form of $X$.
Gel'fand, Kapranov and Zelevinsky characterized coisotropic hypersurfaces by a rank one condition on conormal spaces, which we use as the starting point for our generalization.
We also study the dual notion of isotropic varieties by imposing rank one conditions on tangent spaces instead of conormal spaces.
}

\section{Introduction}
The parametrization of varieties with some fixed properties is one of the most important problems in the history of algebraic geometry.
All subvarieties of $\PP^n$ with fixed dimension and degree can be parametrized by their \emph{Chow forms}.
Chow forms of curves in $\PP^3$ were first introduced by Cayley~\cite{cayley2}.
The generalization to arbitrary projective varieties was given by Chow and van der Waerden~\cite{chow}.

For a given variety $X \subset \mathbb{P}^n$ of dimension $k$, projective subspaces of dimension $n-k-1$ have typically no intersection with the variety $X$, but those subspaces that do intersect $X$ form the \emph{Chow hypersurface} of $X$ in the corresponding Grassmannian~$\Gr(n-k-1, \PP^n)$. 
%The \emph{Chow form} of $X$ is the defining polynomial of this hypersurface. It is a unique (up to scaling with constants) polynomial in the coordinate ring of $\Gr(n-k-1, \PP^n)$ and has the same degree as $X$. 
The \emph{Chow form} of $X$ is an element in the coordinate ring of $\Gr(n-k-1, \PP^n)$ that defines the Chow hypersurface.
It is defined uniquely up to a constant factor and can be represented by a polynomial in the Pl\"ucker coordinates (uniquely modulo the Pl\"ucker relations and scaling with constants).
The Chow form has the same degree as $X$ and one can recover the variety $X$ from its Chow form.
Thus, the variety of Chow forms with a fixed degree in the coordinate ring of the Grassmannian $\Gr(n-k-1, \PP^n)$ is a parameter space for the set of all $k$-dimensional subvarieties $X \subset \PP^n$ with that fixed degree.

\emph{Coisotropic hypersurfaces} are generalizations of Chow hypersurfaces, first introduced by Gel'fand, Kapranov and Zelevinsky \cite[Ch.~4]{gkz}.
A hypersurface in a Grassmannian is called \emph{coisotropic} if its conormal vectors are homomorphisms of rank at most one. Their theorem  \cite[Ch.~4, Thm.~3.14]{gkz} is concerned with the variety of projective linear subspaces of dimension $\ell$ which meet $X$ non-transversely, which may be called the $\ell$-th \emph{associated variety} of $X$. Their proof shows that the class of coisotropic hypersurfaces in $\Gr(\ell, \PP^n)$ is exactly the class of varieties, associated to some $X$, which happen to be hypersurfaces. They show moreover that $X$ may be recovered from any one of its associated hypersurfaces.

Technically, the result \cite[Ch.~4, Thm.~3.14]{gkz} holds only in the dimensional range of $\ell$ for which the associated variety is a hypersurface. Namely $n-k-1 \leq \ell \leq \dim X^{\vee}$, see Theorem~\ref{thm:coisotropicHSdimension}. Although the lower bound $\ell\geq n-k-1$ appears in their statement, the upper bound $\ell\leq \dim X^{\vee}$ is missing.

Our generalization clarifies this issue and dispenses with explicit dimensional restrictions on $\ell$ by invoking a new definition of coisotropy for subvarieties of the Grassmannian of arbitrary codimension, wholly compatible with the definition in codimension 1. This leads to a complete series of associated subvarieties of $X$, one in each Grassmannian $\Gr(\ell,\PP^n)$, which are not necessarily hypersurfaces. 

% 
% I know you have intended to emphasize (and state first) the definition of types of subvarieties in the Grassmannian, rather than their origin/generator subvariety in projective space, but doing so makes it difficult to state the GKZ theorem and your own theorem for comparison. Although the "definition first" approach is shown to be good, it is shown to be good mainly by comparison with the previous, normal way of dealing with coisotropy (i.e. as a criterion for when a Grassmann hypersurface is the image of some X under this transformation).
% -JM

We also study the dual notion of \emph{isotropic varieties} by imposing rank-one-conditions on tangent vectors instead of conormal vectors.
We summarize our results in Subsection~\ref{ssec:results}, pose some open questions in Subsection~\ref{ssec:open}, and list related work in Subsection~\ref{ssec:related}.
In Section~\ref{sec:prelim}, we fix notations and collect various preliminary results.
We prove our results on coisotropic varieties in Section~\ref{sec:coiso}.
Finally, we present our proofs on isotropic varieties in Section~\ref{sec:Iso}.

\subsection{Results}
\label{ssec:results}
We aim to generalize the notion of coisotropic hypersurfaces to subvarieties of Grassmannians having codimension larger than one.
For a subvariety with codimension $c$, the conormal space at a smooth point is a $c$-dimensional vector space of homomorphisms (see Subsection~\ref{ssec:grass}).
There are two natural candidates for a notion of \emph{coisotropic subvarieties} which generalize the case of hypersurfaces:
either we require \emph{every} homomorphism in the conormal spaces of the given variety to have rank at most one, or we require each conormal space to be \emph{spanned} by rank one homomorphisms.
If a subvariety of a Grassmannian satisfies the first condition, we call it \emph{strongly coisotropic}.
If it satisfies the latter condition, we say that the subvariety is \emph{coisotropic}.
As we want our latter condition for coisotropy to be Zariski closed, we need to adjust it slightly:
thus, we call a subvariety of a Grassmannian \emph{coisotropic} if each conormal space is in the Zariski closure of the set of all linear spaces that are spanned by homomorphisms of rank one.
To give a more precise definition in Definition~\ref{defn:coisotropy}, we introduce \emph{Grassmann secant varieties} in Subsection~\ref{ssec:grassSec}.

For the stronger of the two notions, we provide a full geometric characterization, which is a generalization of 
the geometric description of coisotropic hypersurfaces by Gel'fand, Kapranov and Zelevinsky~\cite[Ch.~4, Thm.~3.14]{gkz}:
\begin{thm}
\label{thm:stronglyCoisoFULL}
An irreducible subvariety $\Sigma \subset \Gr(\ell, \mathbb{P}^n)$, not necessarily a hypersurface, is strongly coisotropic if and only if 
there is an irreducible variety $X \subset \PP^n$
such that $\Sigma$ is the Zariski closure of the set of all $\ell$-dimensional projective subspaces 
which intersect $X$ at some smooth point non-transversely. 
%\\ \hspace*{2mm}
%\hfill
\emph{(Proof in Subsection~\ref{ssec:stronglyCoisotropic})}
\end{thm}

\begin{example}
\label{ex:segreStart}
To illustrate the extent of Theorem~\ref{thm:stronglyCoisoFULL} and compare it with~\cite[Ch.~4, Thm.~3.14]{gkz}, 
we consider the four-dimensional Segre variety $X \subset \PP^7$ consisting of $2 \times 4$-matrices of rank one. 
For every $\ell \in \lbrace 0, 1, \ldots, 6 \rbrace$, we can associate a strongly coisotropic variety $\mathcal{G}_\ell(X) \subset \Gr(\ell, \mathbb{P}^7)$ to $X$:
$$ \mathcal{G}_\ell(X)  := \left\lbrace L \in \Gr(\ell, \PP^7) \mid \exists x \in X \cap L : L \text{ intersects } X \text{ at } x \text{ non-transversely}  \right\rbrace. $$ 
As $X$ is smooth, $\mathcal{G}_\ell(X)$ is Zariski-closed.
We will see in Theorem~\ref{thm:coisotropicHSdimension} 
that the variety $\mathcal{G}_\ell(X)$ is a hypersurface in $\Gr(\ell, \mathbb{P}^7)$ if and only if $2 \leq \ell \leq 4$. Note that the upper bound here is $\dim X^{\vee}=4$ (this Segre variety is self-dual).
These are the three coisotropic hypersurfaces of $X$ governed by~\cite[Ch.~4, Thm.~3.14]{gkz}.
Theorem~\ref{thm:stronglyCoisoFULL} comprises all associated subvarieties $\mathcal{G}_0(X), \mathcal{G}_1(X), \ldots, \mathcal{G}_6(X)$ of Grasmannians. We will get back to this example in Example~\ref{ex:segreContinued}.
$\hfill \diamondsuit$
\end{example}

The strongly coisotropic varieties associated to a projective variety $X$ in $\PP^n$ are the strongly coisotropic varieties associated to its projectively dual variety $X^\vee$ in $(\PP^n)^\ast$.
This relation is observed for coisotropic hypersurfaces in \cite[Thm.~20]{coisotropicHS}.
\begin{lem}
\label{lem:dual}
Under the isomorphism $ \Gr(\ell, \mathbb{P}^n) \cong  \Gr(n-\ell-1, (\mathbb{P}^n)^\ast)$, 
the strongly coisotropic variety in $\Gr(\ell, \mathbb{P}^n)$ associated to an underlying variety $X \subset \PP^n$
is the strongly coisotropic variety in $\Gr(n-\ell-1, (\mathbb{P}^n)^\ast)$ associated to the dual variety $X^\vee \subset (\PP^n)^\ast$.
\hfill
\emph{(Proof in Subsection~\ref{ssec:stronglyCoisotropic})}
\end{lem}

For the weaker notion of coisotropy, we first observe two key properties:
\begin{prop}
\label{prop:keyPropCoiso}
\emph{(Proofs in Section~\ref{sec:coiso})}
\begin{enumerate}
\item Every subvariety of $\Gr(\ell, \PP^n)$ of dimension at most $n-1$ is coisotropic.
\item If two coisotropic subvarieties $\Sigma_1, \Sigma_2 \in \Gr(\ell, \PP^n)$ intersect generically transversely, then every irreducible component of $\Sigma_1 \cap \Sigma_2$ is coisotropic.
\end{enumerate}
\end{prop}

\begin{example}
\label{ex:coisotropicInSmallGrass}
Using the first part of Proposition~\ref{prop:keyPropCoiso}, we understand all coisotropic subvarieties of the four-dimensional Grassmannian of lines in $\mathbb{P}^3$:
an irreducible subvariety of $\Gr(1, \PP^3)$ is coisotropic if and only if it is a point, a curve, a surface, or a coisotropic hypersurface.
It was already known by Cayley~\cite{cayley2} that a hypersurface in $\Gr(1,\PP^3)$ is coisotropic if and only if it either consists of all lines intersecting a curve (so it is the Chow hypersurface of this curve), or of all lines tangent to a surface (in which case it is known as the \emph{Hurwitz hypersurface} of the surface, see~\cite{hurwitz}).
Note that this classification of coisotropic hypersurfaces in $\Gr(1, \PP^3)$ is a special case of Theorem~\ref{thm:stronglyCoisoFULL}. 
$\hfill \diamondsuit$
\end{example}

\begin{rem}
\label{rem:selfintersection}
Similarly to the second part of Proposition~\ref{prop:keyPropCoiso}, self-intersections of coisotropic hypersurfaces are coisotropic as well.
For instance, for a $k$-dimensional variety $X \subset\PP^n$, the self-intersection of its Chow hypersurface is the Zariski closure of the set of all projective subspaces of dimension $n-k-1$ which intersect $X$ at two distinct points.
$\hfill \diamondsuit$
\end{rem}

We also show that we need the Zariski closure in the definition of coisotropic varieties.
Indeed, we present a family of geometrically meaningful examples of coisotropic varieties, whose conormal spaces at general points are \emph{not} spanned by homomorphisms of rank one:
for a hypersurface $X \subset \PP^n$ and an integer $m$, we denote by $\mathcal{L}_m(X) \subset \Gr(1, \PP^n)$ the Zariski closure of the set of all lines which intersect $X$ at some smooth point with multiplicity $m$.
\begin{thm}
\label{thm:higherContactLinesFULL}
Let $X \subset \PP^n$ be a general hypersurface of degree at least three
and let $m \in \lbrace 2, 3, \ldots, \min(n, \deg X) \rbrace$.
\begin{enumerate}
\item The subvariety $\mathcal{L}_{m}(X) \subset \Gr(1, \PP^n)$ has codimension $m-1$.
\item The projectivized conormal space at a general point of $\mathcal{L}_{m}(X)$ contains exactly one homomorphism of rank one.
\item Moreover, this projectivized conormal space intersects the Segre variety of all homomorphisms of rank one with multiplicity $m-1$.
\item In particular, $\mathcal{L}_m(X)$ is coisotropic.
$\hfill$ \emph{(Proof in Subsection~\ref{ssec:higherContact})}
\end{enumerate}
\end{thm}
If $m \geq 3$, then $\mathcal{L}_m(X)$ is not a hypersurface in $\Gr(1, \PP^n)$.
By the second part of Theorem~\ref{thm:higherContactLinesFULL}, the conormal spaces of $\mathcal{L}_m(X)$ are \emph{not} spanned by rank one homomorphisms.
However, the third part of Theorem~\ref{thm:higherContactLinesFULL} implies that each conormal space at a smooth point of $\mathcal{L}_m(X)$ is contained in the Zariski closure of the set of all linear spaces which are spanned by homomorphisms of rank one.
Hence, the variety $\mathcal{L}_m(X)$ is coisotropic.

\bigskip 
 
We also introduce a dual notion to coisotropic varieties by imposing rank-one-conditions on the tangent vectors of a subvariety of a Grassmannian instead of on its conormal vectors.
We say that a subvariety of a Grassmannian is \emph{strongly isotropic} if \emph{every} homomorphism in the tangent spaces at its smooth points has rank at most one. 
We call the subvariety \emph{isotropic} if the tangent spaces at its smooth points are \emph{spanned} by rank one homomorphisms, or if they are in the Zariski closure of the set of all linear spaces that are spanned by homomorphisms of rank one. 

Note that both definitions agree for curves in Grassmannians.
In particular, the notion of an \emph{isotropic curve} is dual to the notion of a \emph{coisotropic hypersurface}.
We provide a full geometric characterization of all isotropic curves:
\begin{thm}
\label{thm:isoCurvesFULL}
An irreducible curve $\Sigma \subset \Gr(\ell, \PP^n)$ is isotropic if and only if 
there is a subspace $P \subset \PP^n$ (together with the projection $\pi_P: \mathbb{P}^n \dashrightarrow \PP^{n- \dim P - 1}$ from $P$) and an irreducible curve $C \subset \PP^{n - \dim P -1}$ such that 
$\Sigma$ consists of the preimages of the $(\ell - \dim P - 1)$-dimensional osculating spaces of $C$ under the projection $\pi_P$. \\
$\hspace*{2mm}$ $\hfill$ \emph{(Proof in Subsection~\ref{ssec:isotropicCurves})}
\end{thm}

\begin{rem}
In Theorem~\ref{thm:isoCurvesFULL}, we allow the subspace $P \subset \PP^n$ to be empty. 
In that case, we use the convention that the empty set is a projective space of dimension~$-1$.
$\hfill\diamondsuit$
\end{rem}

\begin{example}
\label{ex:isoCurveStart}
Let $\Sigma$ be an irreducible isotropic curve in the Grassmannian $\Gr(1, \PP^3)$.
By Theorem~\ref{thm:isoCurvesFULL}, the curve $\Sigma$ either consists of the tangent lines of a curve in $\PP^3$, or of the lines on a cone over a plane curve  embedded into $\PP^3$. 
In fact, the curve $\Sigma \subset \Gr(1, \PP^3)$ is isotropic if and only if the surface $\mathcal{S}_\Sigma \subset \PP^3$ ruled by the lines on $\Sigma$ is  \emph{developable}, i.e., the projectively dual variety of $\mathcal{S}_\Sigma$ is a curve or a point. For more details, see Example~\ref{ex:developableSurfacesAreIsotropicCurves}.
$\hfill\diamondsuit$

\end{example}

Dually to Theorem~\ref{thm:stronglyCoisoFULL}, we give a full geometric description of strongly isotropic varieties. 

\begin{thm}
\label{thm:stronglyIsoFULL}
Let $\Sigma \subset \Gr(\ell, \PP^n)$ be an irreducible subvariety of dimension at least two.
The variety $\Sigma$ is strongly isotropic if and only if 
\begin{itemize}
\item either there is an $(\ell+1)$-dimensional subspace $P \subset \PP^n$ such that every $L \in \Sigma$ is contained in $P$, 
\item or there is an $(\ell-1)$-dimensional subspace $P \subset \PP^n$ such that every $L \in \Sigma$ contains $P$. $\hfill$ \emph{(Proof in Subsection~\ref{ssec:stronglyIsotropic})}
\end{itemize}
\end{thm}

\begin{example}
There are no strongly isotropic hypersurfaces in the four-dimensional Grassmannian $\Gr(1, \PP^3)$.
There are exactly two types of irreducible strongly isotropic surfaces:
Such a surface consists either of all lines passing through a fixed point, or of all lines lying in a fixed plane. $\hfill\diamondsuit$
\end{example}

Finally, we observe that subvarieties of Grassmannians of sufficiently high dimension are isotropic.

\begin{prop}
\label{prop:keyPropIso}
Every subvariety of $\Gr(\ell, \PP^n)$ of codimension at most $n-1$ is isotropic.
$\hfill$ \emph{(Proof in Section~\ref{sec:Iso})}
\end{prop}

\begin{example}
All surfaces and hypersurfaces in the four-dimensional Grassmannian $\Gr(1,  \PP^3)$ are isotropic.
Together with Example~\ref{ex:isoCurveStart}, we have classified all isotropic subvarieties of this Grassmannian. 
In particular, surfaces in $\Gr(1, \PP^3)$ are both isotropic and coisotropic.
$\hfill\diamondsuit$
\end{example}

\subsection{Open problems}
\label{ssec:open}
In this article, we provide full geometric characterizations for subvarieties of Grassmannians which are either strongly isotropic or strongly coisotropic.
However, we do not give analogous descriptions for the weaker notions of isotropic or coisotropic varieties, respectively. This motivates the following open problems:

\begin{open} $\hspace*{2mm}$
\begin{enumerate}
\item Can we characterize coisotropic subvarieties of Grassmannians by underlying projective varieties, analogously to Theorem~\ref{thm:stronglyCoisoFULL}?
\item Can we find a geometric characterization of isotropic subvarieties?
\item Is there a comprehensive geometric characterization of subvarieties of Grassmannians which are both isotropic and coisotropic?
\end{enumerate}
\end{open}

Moreover, we conjecture that the irreducible components of the singular locus of a coisotropic variety are coisotropic again.
For instance, the singular locus of a Chow hypersurface of a smooth projective variety is its self-intersection, which is coisotropic by Remark~\ref{rem:selfintersection}.
Another example is given by the variety $\mathcal{L}_2(X)  \subset \Gr(1,\PP^n)$ formed by the tangent lines of a general hypersurface $X \subset \PP^n$.
Its singular locus has two irreducible components: its self-intersection, which is formed by bitangent lines, and $\mathcal{L}_3(X)$, which is coisotropic by Theorem~\ref{thm:higherContactLinesFULL}.
For a study of singular loci of coisotropic hypersurfaces in the Grassmannian of lines in $\PP^3$, see~\cite{congruences}.

\begin{open}
Are the irreducible components of the singular locus of a coisotropic subvariety of a Grassmannian again coisotropic?
\end{open} 

\subsection{Related work}
\label{ssec:related}
In addition to Cayley~\cite{cayley2}, Chow and van der Waerden~\cite{chow}, and Gel'fand, Kapranov and Zelevinsky~\cite{gkz}, many others have studied Chow hypersurfaces and their generalizations to coisotropic hypersurfaces.
For more information on Chow forms, see for instance~\cite{chowIntro}.
An introduction to coisotropic hypersurfaces is presented in~\cite{coisotropicHS}, see Theorem~\ref{thm:coisotropicHSdimension}.

As mentioned above, the variety of Chow forms with a fixed degree in the coordinate ring of the Grassmannian $\Gr(n-k-1, \PP^n)$ is a parameter space for the set of all $k$-dimensional subvarieties of $\PP^n$ with that fixed degree.
This parameter space is called \emph{Chow variety} by Gel'fand, Kapranov and Zelevinsky~\cite[Ch.~4]{gkz}.
A natural question is to describe the vanishing ideal of such a Chow variety, which was already investigated by Green and Morrison~\cite{green_morr}.
This ideal is explicitly computed for the variety of all quadratic Chow hypersurfaces in $\Gr(1, \PP^3)$ in~\cite{our_computation}.

A first type of coisotropic hypersurface which is not a Chow hypersurface is studied by Sturmfels in~\cite{hurwitz}: given an irreducible subvariety $X \subset \PP^n$ of degree at least two and dimension $k$, the \emph{Hurwitz hypersurface} of $X$ is the subvariety of the Grassmannian $\Gr(n-k, \PP^n)$ consisting of all $(n-k)$-dimensional subspaces which do not intersect $X$ in $\deg(X)$ reduced points.
As mentioned in Example~\ref{ex:coisotropicInSmallGrass}, Cayley~\cite{cayley2} and Gel'fand, Kapranov and Zelevinsky~\cite[Ch.~4]{gkz}  knew that a hypersurface of $\Gr(1, \PP^3)$ is coisotropic if and only if it is either the Chow hypersurface of a curve or the Hurwitz hypersurface of a surface.
Recently, Catanese~\cite{catanese} showed that these are exactly the self-dual hypersurfaces of $\Gr(1, \PP^3)$.
Moreover, the volume of the Hurwitz hypersurface of a projective variety $X$ plays a crucial role in the study of the condition of intersecting $X$ with varying linear subspaces~\cite{condition}.
More generally, one can compute the volume of all coisotropic hypersurfaces, which is essential for the probabilistic Schubert calculus proposed by B\"urgisser and Lerario~\cite{schubert}.

\section{Preliminaries}
\label{sec:prelim}
We work over an algebraically closed field $\KK$ of characteristic zero.
All vector spaces are finite-dimensional.
The dual vector space of a given vector space $V$ is denoted by $V^\ast$.

\subsection{Varieties}
All algebraic varieties that appear in this article are either affine or projective.
Hence, we understand a variety as the common zero locus of (homogeneous) polynomials.
In particular, varieties do not necessarily have to be irreducible.
We say that an affine or projective variety is \emph{nondegenerate} if it spans its ambient space.

We denote the affine cone over a subset $S \subset \PP^n$ by $\aff{S} \subset \KK^{n+1}$.
The projectivization of an affine cone $\mathcal{C}$ is denoted by $\PP(\mathcal{C})$.
We write $\Sing(X)$ for the singular locus of a variety $X$ and $\Reg(X) := X \setminus \Sing(X)$ for its regular locus.
If $X \subset \PP^n$, the \emph{embedded tangent space} of $X$ at $p \in \Reg(X)$ is
%\begin{align}
%\label{eq:embeddedTangentSpace}
%\TT_{X,p} := \left\lbrace
%y \in \PP^n \;\middle\vert\; \forall f \in I(X) : \sum_{i=0}^n \frac{\partial f}{\partial x_i}(p) \cdot y_i =0
%\right\rbrace.
%\end{align}

 \begin{align}
 \label{eq:embeddedTangentSpace}
 \begin{split}
 \TT_{X,p} :=  & \left\lbrace
 y \in \PP^n \;\middle\vert\; \forall f \in I(X) : \langle df|_{\aff{p}} , \aff{y} \rangle =0
 \right\rbrace.\\
 = & \left\lbrace
 y \in \PP^n \;\middle\vert\; \forall f \in I(X), \forall v \in \aff{p}, \forall w \in \aff{y} : \langle df|_{v} , w \rangle =0
 \right\rbrace.
 \end{split}
 \end{align}

A hyperplane in $\PP^n$ is called \emph{tangent} to $X$ at a smooth point $p$ if it contains $\TT_{X,p}$.
Note that each hyperplane $H \subset \PP^n$ corresponds to a point in $(\PP^n)^\ast:=\PP((\KK^{n+1})^{*})$, namely $\PP(\operatorname{Ann}(\aff{H}))$, which we denote by $H^\vee$.
The \emph{projectively dual variety} of $X$ is
\begin{align*}
X^\vee := \overline{ \left\lbrace
H^\vee \mid H \subset \PP^n \text{ hyperplane}, \exists p \in \Reg(X) : \TT_{X,p} \subset H
\right\rbrace} \subset (\PP^n)^\ast.
\end{align*}
If $X$ is irreducible, so is $X^\vee$~\cite[Ch.~1, Prop.~1.3]{gkz}.
For example, if $L$ is an $\ell$-dimensional projective subspace of $\PP^n$, then $L^\vee = \PP(\operatorname{Ann}(\aff{L}))$ is a projective subspace of $(\PP^n)^\ast$ with dimension $n-\ell-1$, whose points correspond to hyperplanes containing~$L$. 
In particular, we have $(\PP^n)^\vee = \emptyset$.
Throughout this article, we use the convention that the empty set is a projective space with dimension $-1$.
We will make frequent use of the following \emph{biduality} of projective varieties, which is also known as \emph{reflexivity}.

\begin{thm}[Biduality theorem {\cite[Thm.~4]{wallace}}]
Let $X \subset \PP^n$ be a projective variety over an algebraically closed field of characteristic zero.
If $p \in \Reg(X)$ and $H^\vee \in \Reg(X^\vee)$, then the hyperplane $H$ is tangent to $X$ at the point $p$ if and only if the hyperplane $p^\vee$ is tangent to $X^\vee$ at the point $H^\vee$.
In particular, we have $(X^\vee)^\vee = X$.
$\hfill \diamondsuit$
\end{thm}

\subsection{Grassmannians}
\label{ssec:grass}
The \emph{Grassmannian} $\Gr(\ell, \PP^n)$ is the set of all projective subspaces of $\PP^n$ with dimension $\ell$.
It is naturally identified with the Grassmannian $\Gr(\ell+1, \KK^{n+1})$ of $(\ell+1)$-dimensional linear subspaces of $\KK^{n+1}$.
Moreover, it is an irreducible smooth $(\ell+1)(n-\ell)$-dimensional variety, embedded in $\PP^{\binom{n+1}{\ell+1}-1}$ via the Pl\"ucker embedding.

The tangent space of $\Gr(\ell+1, \KK^{n+1})$ at a point $\aff{L} \in \Gr(\ell+1, \KK^{n+1})$ is naturally identified with
\begin{align}
\label{eq:tangentSpaceGrassmannian}
T_{\Gr(\ell+1, \KK^{n+1}), \aff{L}} := & \; \Hom(\aff{L}, \KK^{n+1} / \aff{L}) \\
 \cong & \; \aff{L}^{\ast}\otimes (\KK^{n+1}/\aff{L}). \nonumber
\end{align}
A detailed discussion of tangent spaces of Grassmannians can be found in, for example, \cite[Lecture~16]{harris}. 
Throughout this article we closely follow the $\Hom$ convention used in \cite[Lecture~16]{harris}.
In fact, the description in~\eqref{eq:tangentSpaceGrassmannian} is very intuitive.
A tangent vector at $\aff{L}$ is a direction in which $\aff{L}$ can move.
Such a direction can be specified by describing the movement of every point on the linear subspace $\aff{L}$, i.e., by a homomorphism $\aff{L} \to \KK^{n+1}$.
If the direction of movement of a point on $\aff{L}$ lies inside $\aff{L}$, its movement does not contribute to the direction of movement of $\aff{L}$.
Thus, a tangent vector of $\Gr(\ell+1,\KK^{n+1})$ at $\aff{L}$ is already given by a homomorphism $\varphi: \aff{L} \to \KK^{n+1} / \aff{L}$.

\begin{rem}
Since $T_{\Gr(\ell, \PP^n),{L}} = \Hom(\aff{L}, \KK^{n+1} / \aff{L})$, as a special case the tangent space of projective space at a point $p \in \PP^n$ is
\begin{align*}
T_{\PP^n,p} = \Hom(\aff{p}, \KK^{n+1} / \aff{p}).
\end{align*}
% I got rid of the isomorphism we had that drops the factor of p* (non-canonical).
Hence, for a projective variety $X \subset \PP^n$, we distinguish between the \emph{Zariski tangent space} $T_{X,p}$ at $p \in \Reg(X)$ and the embedded tangent space $\TT_{X,p} \subset \PP^n$; see~\eqref{eq:embeddedTangentSpace}.
Note that the dimension of both tangent spaces is $\dim(X)$, although the embedded tangent space is a projective space, whereas the Zariski tangent space is an abstract vector space. 
For an affine variety $Y \subset \KK^{n+1}$, 
we only consider the Zariski tangent space $T_{Y,p} \subset \KK^{n+1}$ at $p \in \Reg(Y)$.
Note that
$T_{X,p} = \Hom(\aff{p}, T_{\aff{X},x} / \aff{p})$ and
$\TT_{X,p} = \PP(T_{\aff{X},x})$ for $x \in \aff{p} \setminus \lbrace 0 \rbrace$.
$\hfill\diamondsuit$
\end{rem}

For a subvariety $\Sigma \subset \Gr(\ell+1, \KK^{n+1})$, the \emph{normal space} of $\Sigma$ at $\aff{L} \in \Reg(\Sigma)$ is $N_{\Sigma, \aff{L}} := T_{\Gr(\ell+1, \KK^{n+1}), \aff{L}} / T_{\Sigma, \aff{L}}$.
Its dual vector space is the \emph{conormal space} of $\Sigma$ at~$\aff{L}$.
For two vector spaces $U$ and $W$, we identify the dual vector space of $\Hom(U,W)$ with $\Hom(W,U)$ via
% - Assignment HW (mention tensor-notation duality)
\begin{align*}
\Hom(W,U) &\longrightarrow \Hom(U,W)^\ast, \\
\phi &\longmapsto \mathrm{tr}(\cdot \circ \phi) = \mathrm{tr}(\phi \circ \cdot)
\end{align*}
Equivalently, using tensor notation,
\begin{align*}
W^{\ast}\otimes U &\longrightarrow (U^{\ast}\otimes W)^\ast, \\
l\otimes u &\longmapsto (k\otimes w\mapsto l(w)k(u))
\end{align*}
This identification is used to express the conormal space of $\Sigma$ at $\aff{L}$:
\begin{align*}
\left(N_{\Sigma, \aff{L}} \right)^\ast
&\isom \left\lbrace \varphi \in \left( T_{\Gr(\ell+1, \KK^{n+1}), \aff{L}} \right)^\ast \mid T_{\Sigma, \aff{L}} \subset \ker \varphi \right\rbrace \\
&\isom \left\lbrace \varphi \in \Hom(\KK^{n+1} / \aff{L}, \aff{L}) \mid
\forall \psi \in T_{\Sigma,\aff{L}} : \mathrm{tr}(\varphi \circ \psi)=0 
\right\rbrace \\
& =: N_{\Sigma,\aff{L}}^\ast
\end{align*}
An advantage of using \emph{co}normal spaces is that conormal vectors can be viewed as homomorphisms by these identifications, whereas normal vectors are cosets, elements of a quotient space.
In particular, each conormal vector has a rank.

\begin{remDef}
\label{rem:dualityConormalSpaces}
There is a canonical isomorphism 
\begin{align*}
\begin{split}
\Gr(\ell, \PP^n) &\overset{\sim}{\longrightarrow} \Gr(n-\ell-1, (\PP^n)^\ast), \\
L &\longmapsto L^\vee.
\end{split}
\end{align*}
Using vector spaces it can be expressed as follows:
\begin{align*}
\Gr(\ell+1, \KK^{n+1}) &\overset{\sim}{\longrightarrow} \Gr(n-\ell, (\KK^{n+1})^\ast),\\
\aff{L} &\longmapsto \operatorname{Ann}(\aff{L})\cong\left( \KK^{n+1} / \aff{L} \right)^\ast.
\end{align*}
The differential of this isomorphism at $\aff{L} \in \Gr(\ell+1, \KK^{n+1})$ is
\begin{align*}
\Hom \left(\aff{L}, \KK^{n+1} / \aff{L} \right) &\longrightarrow \Hom \left(\left( \KK^{n+1} / \aff{L} \right)^\ast, \aff{L}^\ast \right),\\
\varphi &\longmapsto \varphi^\ast.
\end{align*}
For a subvariety $\Sigma \subset \Gr(\ell+1, \KK^{n+1})$, we denote by $\Sigma^\perp \subset \Gr(n-\ell, (\KK^{n+1})^\ast)$
the image of $\Sigma$ under this isomorphism.
Since the trace of an endomorphism $\psi$ is zero if and only if the trace of $\psi^\ast$ is zero,
the isomorphism of the conormal space of $\Sigma$ at $\aff{L} \in \Reg(\Sigma)$ with the corresponding conormal space of $\Sigma^\perp$ is
\begin{align*}
\hspace{25mm}
\begin{array}{cccc}
 \Hom \left(\KK^{n+1} / \aff{L}, \, \aff{L} \right) && \Hom \left( \aff{L}^\ast, \left( \KK^{n+1} / \aff{L} \right)^\ast \right) \\
 \cup && \cup \\
N^\ast_{\Sigma, \aff{L}} & \longrightarrow & N^\ast_{\Sigma^\perp, (\KK^{n+1} / \aff{L} )^\ast}, \\
 \psi & \longmapsto & \psi^\ast.
&
\hspace{24mm} \diamondsuit
\end{array}
\end{align*}
\end{remDef} 

In this article, we will often consider incidence correspondences of Grassmannians and their tangent spaces.
For this, we introduce the following notation.
Let $\varphi: U \to V_1 / V_2$ be a homomorphism of vector spaces, where $V_2 \subset V_1$.
For a subspace $W \subset U$, we denote by $\varphi |_W: W \to V_1 / V_2$ the restriction of $\varphi$ to $W$.
Furthermore, for a vector space $V$ with $V_2 \subset V \subset V_1$ we denote by $(\varphi \mod V): U \to  V_1 / V$ the composition of the canonical projection $V_1 / V_2 \to V_1 / V$ with $\varphi$.

\begin{lem}
\label{lem:tangentSpaceIncidenceGrassmannian}
Let $k < \ell$.
The tangent space of the \emph{flag variety}
\begin{align*}
F_{k,\ell} := \left\lbrace (\aff{L_1}, \aff{L_2}) \in \Gr(k+1, \KK^{n+1}) \times \Gr(\ell+1, \KK^{n+1}) \mid \aff{L_1} \subset \aff{L_2} \right\rbrace
\end{align*}
at a point $(\aff{L_1}, \aff{L_2}) \in F_{k, \ell}$ is
\begin{align*}
T_{F_{k,\ell},(\aff{L_1}, \aff{L_2})} &= \left\lbrace
(\psi, \varphi)  \mid \varphi |_{\aff{L_1}} = (\psi \mod \aff{L_2})
\right\rbrace \\
&\subset \Hom(\aff{L_1}, \KK^{n+1} / \aff{L_1}) \times \Hom(\aff{L_2}, \KK^{n+1} / \aff{L_2}).
\end{align*}
\end{lem}

\begin{proof}
From our description of tangent spaces to Grassmannians it follows immediately that every tangent vector $(\psi, \varphi) \in T_{F_{k,\ell},(\aff{L_1}, \aff{L_2})}$ and every $v \in \aff{L_1}$ satisfy
$\varphi(v) = \psi(v)$ modulo $\aff{L_2}$; see also~\cite[Example~16.2, Exercise~16.3]{harris}.
Hence, $T_{F_{k,\ell},(\aff{L_1}, \aff{L_2})} \subset \lbrace (\psi, \varphi)  \mid \varphi |_{\aff{L_1}} = (\psi \mod \aff{L_2}) \rbrace$.
Since both vector spaces have the same dimension, they are equal.
\end{proof}

We usually apply this result to graphs of rational maps which are contained in~$F_{k,\ell}$.
Therefore, we formulate this version of the result explicitly in Corollary~\ref{cor:keyLemmaTangentCorrespondence}.
For a rational map $\varpi: X \dashrightarrow Y$ between varieties, we denote by $D_p \varpi: T_{X,p} \to T_{Y, \varpi(p)}$ the differential of $\varpi$ at $p$ where it is defined.

\begin{example}
\label{ex:tangentLinesToCurveRationalMap}
Many proofs in this article will follow the ideas presented in this example.
Let $C \subset \PP^3$ be an irreducible curve of degree at least two. The graph of the rational map
\begin{align*}
\varpi: C &\longdashrightarrow \Gr(1, \PP^3),\\
\Reg(C) \ni p &\longmapsto \TT_{C,p}
\end{align*}
is contained in the flag variety $F_{0,1}$.
We denote the Zariski closure of the image of $\varpi$ by $\mathcal{T}(C)$.
This curve of tangent lines appears in Examples~\ref{ex:isoCurveStart} and~\ref{ex:developableSurfacesAreIsotropicCurves}.
For a general point $p$ on the curve $C$ with tangent line $L := \TT_{C,p}$, the differential $D_p \varpi: T_{C,p} \to T_{\mathcal{T}(C), L}$ is an isomorphism.
Note that $$T_{C,p} = \lbrace \varphi \in \Hom(\aff{p}, \KK^4 / \aff{p}) \mid \im \varphi \subset \aff{L} / \aff{p} \rbrace.$$
By the following Corollary~\ref{cor:keyLemmaTangentCorrespondence}, we have for every $\varphi \in T_{C,p}$ that
\begin{align*}
D_p \varpi (\varphi) |_{\aff{p}} = (\varphi \mod \aff{L}).
\end{align*}
Since the image of every $\varphi \in T_{C,p}$ is contained in $\aff{L} / \aff{p}$ and $D_p \varpi$ is an isomorphism, 
the kernel of every homomorphism in $T_{\mathcal{T}(C), L} \subset \Hom(\aff{L}, \KK^4 / \aff{L})$ contains $\aff{p}$.
In particular, the one-dimensional vector space $T_{\mathcal{T}(C), L}$ is spanned by a homomorphism $\psi: \aff{L} \to \KK^4 / \aff{L}$ such that $\aff{p} \subset \ker \psi$.
Since $\aff{L}$ is a plane containing the line $\aff{p}$, the kernel of $\psi$ is equal to $\aff{p}$ and the homomorphism $\psi$ has rank one.
We will see in Subsection~\ref{ssec:isotropicCurves} that the one-dimensional image of this homomorphism is $\aff{H}/\aff{L}$, where $H \subset \PP^3$ is the \emph{osculating plane} of the curve $C$ at the point $p$.
$\hfill\diamondsuit$
\end{example}

\begin{cor}
\label{cor:keyLemmaTangentCorrespondence}
Let $\Sigma \subset \Gr(\ell+1, \KK^{n+1})$ be an irreducible subvariety with a rational map $\varpi: \Sigma \dashrightarrow \Gr(k+1, \KK^{n+1})$.
If $k \leq \ell$ and $\varpi(\aff{L}) \subset \aff{L}$ for a general $\aff{L} \in \Sigma$, then a general $\aff{L} \in \Sigma$ satisfies
\begin{align}
\label{eq:tangentCorrespondence}
\varphi |_{{\varpi(\aff{L})}} = \left( D_{\aff{L}} \varpi (\varphi)  \mod {\aff{L}} \right)
\end{align}
for every $\varphi \in T_{\Sigma, \aff{L}}$, i.e., the following diagram commutes:
\begin{center}
\begin{small}
\begin{tikzcd}[column sep=tiny, framed]
\Hom \left( \aff{L}, \KK^{n+1}/\aff{L} \right) \supset \hspace*{-5mm} & T_{\Sigma, \aff{L}}  \arrow[dr, "\cdot |_{\varpi(\aff{L})}" {below, near start}] \arrow[rr, "D_{\aff{L}} \varpi"] & & \Hom \left(\varpi(\aff{L}), \, \KK^{n+1} / \varpi(\aff{L}) \right) \arrow[dl, "(\cdot \mod \aff{L})" near start] \\ 
& & \Hom \left(\varpi(\aff{L}), \, \KK^{n+1}/\aff{L} \right)  & 
\end{tikzcd}
\end{small}
\end{center}
If $k \geq \ell$ and $\varpi(\aff{L}) \supset \aff{L}$ for a general $\aff{L} \in \Sigma$, then a general $\aff{L} \in \Sigma$ satisfies
\begin{align*}
D_{\aff{L}} \varpi(\varphi) |_{\aff{L}} = \left( \varphi \mod {\varpi(\aff{L})} \right)
\end{align*}
for every $\varphi \in T_{\Sigma, \aff{L}}$, i.e., the following diagram commutes:
\begin{center}
\begin{small}
\begin{tikzcd}[column sep=tiny, framed]
\Hom \left( \aff{L}, \KK^{n+1}/\aff{L} \right) \supset \hspace*{-5mm} & T_{\Sigma, \aff{L}}  \arrow[dr, "(\cdot \mod \varpi(\aff{L})) \hspace*{16mm}" {below, near start}] \arrow[rr, "D_{\aff{L}} \varpi"] & & \Hom \left(\varpi(\aff{L}), \, \KK^{n+1} / \varpi(\aff{L}) \right) \arrow[dl, "\cdot |_{\aff{L}}" near start] \\ 
& & \Hom \left(\aff{L}, \, \KK^{n+1}/\varpi(\aff{L}) \right)  & 
\end{tikzcd}
\end{small}
\end{center}
\end{cor}

\begin{proof}
We denote by $\Sigma_0 \subset \Sigma$ the subset where $\varpi$ is defined.
Let us first assume that $k \leq \ell$ and that $\varpi(\aff{L}) \subset \aff{L}$ holds for all $\aff{L} \in \Sigma_0$.
We consider the Zariski closure of the graph of $\varpi$:
\begin{align*}
\Gamma := \overline{\left\lbrace
(\varpi(\aff{L}), \aff{L}) \mid \aff{L} \in \Sigma_0
\right\rbrace} \subset \left( \Gr(k+1, \KK^{n+1}) \times \Sigma \right) \cap F_{k,\ell}.
\end{align*}
Moreover, let $\pi_1 :\Gamma \to \Gr(k+1, \KK^{n+1})$ and $\pi_2: \Gamma \to \Sigma$ denote the projections onto the first and second factor, respectively.
We can draw another commutative diagram (see below):
at a general $\aff{L} \in \Sigma$, the differential $D_{(\varpi(\aff{L}),\aff{L})} \pi_2$ is bijective and
$D_{\aff{L}} \varpi =  D_{(\varpi(\aff{L}),\aff{L})} \pi_1 \circ (D_{(\varpi(\aff{L}),\aff{L})} \pi_2)^{-1}$.
Hence, by Lemma~\ref{lem:tangentSpaceIncidenceGrassmannian}, the differential $D_{\aff{L}} \varpi$ of $\varpi$ at $\aff{L}$ satisfies~\eqref{eq:tangentCorrespondence}.
\begin{center}
\begin{small}
\begin{tikzcd}[framed]
T_{\Gamma, \, (\varpi(\aff{L}),\aff{L})}
\arrow[d, hook, two heads, "D_{(\varpi(\aff{L}),\aff{L})} \pi_2" left]
\arrow[r, "D_{(\varpi(\aff{L}),\aff{L})} \pi_1"] &[16mm] \Hom \left( \varpi(\aff{L}), \, \KK^{n+1}/\varpi(\aff{L}) \right)  \\
T_{\Sigma, \aff{L}}
\arrow[ru, "\hspace*{4mm} D_{\aff{L}} \varpi" below]
\end{tikzcd}
\end{small}
\end{center}

If $k \geq \ell$ and $\aff{L} \subset \varpi(\aff{L})$ for all $\aff{L} \in \Sigma_0$, we proceed analogously.
\end{proof}

With Lemma~\ref{lem:tangentSpaceIncidenceGrassmannian} we can also understand the differential of the rational map
\begin{align}
\label{eq:StiefelCoords}
\begin{split}
\rho: \KK^{(\ell+1) \times (n+1)} &\longdashrightarrow \Gr(\ell+1, \KK^{n+1}), \\
A &\longmapsto \mathrm{rowspace}(A).
\end{split}
\end{align}
For an $(\ell+1) \times (n+1)$-matrix $A$ and $0 \leq i \leq \ell$, we write $A_i \in \KK^{n+1}$ for the $i$-th row of $A$.

\begin{cor}
\label{cor:differentialStiefel}
Let $A \in \KK^{(\ell+1) \times (n+1)}$ be a matrix of full rank and let \mbox{$\aff{L} := \rho(A)$} denote its rowspace. 
The differential of~\eqref{eq:StiefelCoords} at $A$ is given by
\begin{align*}
D_A \rho: \KK^{(\ell+1) \times (n+1)} &\longrightarrow \Hom(\aff{L}, \KK^{n+1} / \aff{L}), \\
M &\longmapsto \left( A_i \mapsto M_i + \aff{L} \quad\quad \text{ for all } 0 \leq i \leq \ell \right).
\end{align*}
\end{cor}

\begin{proof}
Let $i \in \lbrace 0, \ldots, \ell \rbrace$.
We consider the map 
\begin{align*}
\pi_i: \KK^{(\ell+1) \times (n+1)} &\longdashrightarrow \Gr(1, \KK^{n+1}) \isom \PP^n, \\
A &\longmapsto \spann(A_i).
\end{align*}
Its differential at $A$ is
\begin{align*}
D_A \pi_i: \KK^{(\ell+1) \times (n+1)} &\longrightarrow \Hom \left(\spann(A_i), \, \KK^{n+1} / \spann(A_i) \right), \\
M &\longmapsto \left( A_i \mapsto M_i + \spann(A_i) \right).
\end{align*}
The differential of the product map $\pi_i \times \rho: \KK^{(\ell+1) \times (n+1)} \dashrightarrow F_{0, \ell}$ at $A$
is equal to $D_A \pi_i \times D_A \rho$.
Hence, by Lemma~\ref{lem:tangentSpaceIncidenceGrassmannian}, we have
for every $M \in \KK^{(\ell+1) \times (n+1)}$ that
$$(D_A \rho (M))(A_i) = (D_A \pi_i (M) \mod \aff{L})(A_i) = M_i + \aff{L}. $$
\end{proof}

\subsection{Coisotropic hypersurfaces}
Gel'fand, Kapranov and Zelevinsky prove Theorem~\ref{thm:stronglyCoisoFULL} for hypersurfaces in Grassmannians~\cite[Ch.~4, Thm.~3.14]{gkz}.
An alternative proof can be found in~\cite{coisotropicHS}.
\begin{defn} \label{def:coisotropicHS} $\hspace*{2mm}$
\begin{enumerate}
\item A hypersurface $\Sigma \subset \Gr(\ell, \PP^n)$ is called \emph{coisotropic} if, for every $L \in \Reg(\Sigma)$, the conormal space $N_{\Sigma, L}^\ast \subset \Hom(\KK^{n+1} / \aff{L}, \aff{L})$ is spanned by a homomorphism of rank one. 
\item The $\ell$-th \emph{higher associated variety} $\mathcal{G}_\ell(X) \subset \Gr(\ell, \PP^n)$ of an irreducible subvariety $X \subset \PP^n$ consists of all $\ell$-dimensional subspaces which intersect $X$ at some smooth point non-transversely:
\begin{align*}
\mathcal{G}_\ell(X) := \overline{\left\lbrace  L \mid   \exists x \in \Reg(X) \cap L : \dim(L + \TT_{X,x})  < n \right\rbrace} \subset \Gr(\ell, \PP^n).
\end{align*}
\end{enumerate}
\end{defn}

\noindent
Here and throughout the rest of this article $L_1+L_2$ denotes the projective span of two projective subspaces $L_1$ and $L_2$.

With Definition~\ref{def:coisotropicHS}, we can restate Theorem~\ref{thm:stronglyCoisoFULL} for hypersurfaces:
an irreducible hypersurface $\Sigma \subset \Gr(\ell, \PP^n)$ is coisotropic if and only if there is an irreducible variety $X \subset \PP^n$ such that $\Sigma = \mathcal{G}_\ell(X)$.
We also note that the underlying projective variety $X$ can be uniquely recovered from its higher associated hypersurfaces $\mathcal{G}_\ell(X)$. 
In particular, there are two \texttt{Macaulay2} packages which compute coisotropic hypersurfaces and recover their underlying projective varieties: \texttt{Coisotropy.m2} and \texttt{Resultants.m2}.

\begin{rem}
The higher associated varieties of a projective variety  interpolate between the variety itself and and its projectively dual variety:
for $X \subset \PP^n$, we have $\mathcal{G}_0(X) = X$ and $\mathcal{G}_{n-1}(X) = X^\vee$.
$\hfill\diamondsuit$
\end{rem}

It is known for which $\ell$ the higher associated varieties $\mathcal{G}_\ell(X)$ are hypersurfaces in their ambient Grassmannian. 
Moreover, the degrees of the higher associated hypersurfaces of a projective variety are exactly its well-studied \emph{polar degrees}:

For an irreducible subvariety $X \subset \PP^n$, an integer $\ell \in \lbrace 0, 1, \ldots, n-1 \rbrace$, and a projective subspace $V \subset \PP^n$ of dimension $\ell-1$, the $\ell$-th \emph{polar variety} of $X$ with respect to $V$ is
$$
P_\ell(X,V) := \overline{\left\lbrace x \in \Reg(X) \mid \dim (V \cap \TT_{X,x}) \geq \ell - \codim X \right\rbrace} \subset X.
$$
There is an integer $\delta_\ell(X)$ which is equal to the degree of $P_\ell(X,V)$ for almost all $V$. This integer $\delta_\ell(X)$ is known as the $\ell$-th \emph{polar degree} of $X$. 
It is strictly positive if and only if $\codim X - 1 \leq \ell \leq \dim X^\vee$.
These and many other properties of polar degrees can be found in \cite{piene} or \cite{holme}.

\begin{thm}[{\cite{coisotropicHS}}]
\label{thm:coisotropicHSdimension}
Let $X \subset \PP^n$ be an irreducible subvariety.
\begin{enumerate}
\item The $\ell$-th higher associated variety $\mathcal{G}_\ell(X)$ has codimension one in $\Gr(\ell, \PP^n)$ if and only if $\codim X - 1 \leq \ell \leq \dim X^\vee$. $\quad\quad\quad\quad\quad\quad\quad\quad\quad$
(Note that this coincides with the range of indices where the polar degrees of $X$ are positive.)
\item If $\delta_\ell(X) > 0$, then $\delta_\ell(X) = \deg \mathcal{G}_\ell(X)$.
\end{enumerate}
\end{thm}

\begin{example}
The Chow hypersurface of a subvariety $X \subset \PP^n$ is $\mathcal{G}_{\codim X - 1}(X)$. Its degree equals the degree of $X$.
$\hfill\diamondsuit$
\end{example}

\subsection{Segre varieties}
Since the elements of the tangent and conormal spaces of subvarieties of Grassmannians are homomorphisms, we can consider their rank. 
In this article, homomorphisms of rank one play a central role.
For two vector spaces $U$ and $W$, we denote by $\Seg(U,W)$ the projectivization of $\lbrace \varphi \in \Hom(U,W) \mid \rank\,\varphi \leq 1 \rbrace$. 
This \emph{Segre variety} has dimension $d+e$ and degree $\binom{d+e}{d}$, where $d := \dim U - 1$ and $e := \dim W-1$.
It has two rulings by maximal projective subspaces:
one ruling consists of the projectivizations of all \emph{$\alpha$-spaces},
and the other ruling of the projectivizations of all \emph{$\beta$-spaces}, which are defined as follows:

\begin{defn}
\label{def:alphaBetaSpace}
Let $U$ and $W$ be vector spaces.
For a linear hyperplane $u \subset U$, we define the \emph{$\alpha$-space} of $u$ as
\begin{align*}
E_\alpha(u) := \lbrace \varphi \in \Hom(U,W) \mid u \subset \ker \varphi \rbrace.
\end{align*}
Analogously, the \emph{$\beta$-space} of a one-dimensional linear subspace $w \subset W$ is
\begin{align*}
E_\beta(w) := \lbrace \varphi \in \Hom(U,W) \mid \im \varphi \subset w \rbrace.
\end{align*}
\end{defn}

\begin{example}
Let $U := \CC^4$ and $W := \CC^3$. We pick bases $ (e_1, \ldots, e_4 )$ and $(f_1, \ldots, f_3)$ of $U$ and $W$, respectively. With respect to these bases we identify $\Hom(U,W)$ with $\CC^{3 \times 4}$.
For $u := \spann \lbrace e_1, e_2, e_3 \rbrace$ and $w := \spann \lbrace f_1 \rbrace$, we have that $E_\alpha(u)$ and $E_\beta(w)$ consist of all matrices of the form
\begin{align*}
\left[ \begin{array}{cccc}
0 & 0 & 0 & \ast \\
0 & 0 & 0 & \ast \\
0 & 0 & 0 & \ast 
\end{array} 
\right]
\quad\quad \text{ resp. } \quad\quad
\left[ \begin{array}{cccc}
\ast & \ast & \ast & \ast \\
0 & 0 & 0 & 0 \\
0 & 0 & 0 & 0 
\end{array} 
\right].
\end{align*}

\vspace*{-8mm}
$\hfill\diamondsuit$
\end{example}

\begin{rem}
The isomorphism 
\begin{align*}
\Hom(U,W) &\longrightarrow \Hom(W^\ast, U^\ast), \\
\varphi &\longmapsto \varphi^\ast
\end{align*}
maps the $\alpha$-space $E_\alpha(u)$ to the $\beta$-space $E_\beta((U/u)^\ast)$.
Dually, it maps the $\beta$-space $E_\beta(w)$ to the $\alpha$-space $E_\alpha((W/w)^\ast)$.
$\hfill\diamondsuit$
\end{rem}

Note that the notions \emph{rank}, \emph{kernel} and \emph{image} are well-defined for elements in $\PP(\Hom(U,W))$.
Thus, for $\varphi \in \PP(\Hom(U,W))$, we will usually write $\rank \, \varphi$, $\ker \varphi$ and $\im \varphi$, respectively.

\subsection{Grassmann secant varieties}
\label{ssec:grassSec}

We define the following analogue of secant varieties in Grassmannians.
For an irreducible variety $X \subset \PP^n$ and $k \in \ZZ_{\geq 0}$, we set
\begin{align*}
\Sec^0_k(X) := \left\lbrace
L \in \Gr(k, \PP^n) \mid L = \spann (L \cap X)
\right\rbrace
\end{align*}
and let $\Sec_k(X) \subset \Gr(k, \PP^n)$ denote its Zariski closure, called the \emph{$k$-th Grassmann secant variety of $X$}.
Grassmann secant varieties have recently played a role in the study of tensor rank and Waring's problem~\cite{landsberg, waring1, waring2, waring3}, 
but they have also been studied on their own right~\cite{grSec, grSecDim, grSecSing}.
In~\cite{grSecDim}, Ciliberto and Cools show, for every irreducible and non-degenerate variety $X \subset \PP^n$, that
\begin{align}
\label{eq:dimGrSec}
\dim \Sec_k(X) \geq (k+1) \dim X, \quad\quad\quad \text{if } \Sec_k(X) \neq \Gr(k, \PP^n).
\end{align}
This implies immediately the following.

\begin{lem}
\label{lem:grSecSpans}
Let $X \subset \PP^n$ be an irreducible nondegenerate variety.
If $k \geq \codim X$, then $\Sec_k(X) = \Gr(k, \PP^n)$.
\end{lem}

\begin{proof}
If $k \geq \codim X$,
then $\dim X \geq n-k$, so
$(k+1) \dim X \geq \dim \Gr(k, \PP^n)$ and the assertion follows from~\eqref{eq:dimGrSec}.
\end{proof}

We need one more key property of Grassmann secant varieties for our purposes.

\begin{lem}
\label{lem:grSecSum}
Consider an irreducible variety $X \subset \PP^n$.
For $L_1 \in \Sec_{k_1}(X)$ and \linebreak[4] $L_2 \in \Sec_{k_2}(X)$, we have that $L_1 + L_2 \in \Sec_{\dim(L_1+L_2)} (X)$.
\end{lem}

\begin{proof}
Let $k := \dim (L_1 + L_2)$.
For $i \in \lbrace 1,2 \rbrace$, we write $S_i := \Sec_{k_i}(X)$ and $S_i^0 := \Sec^0_{k_i}(X)$.
In addition, we set $$U := \lbrace (L'_1, L'_2) \in S^0_1 \times S^0_2 \mid \dim (L'_1 + L'_2) = k \rbrace.$$
For every $(L'_1, L'_2) \in U$, we have $L'_1+L'_2 \in \Sec_k(X)$. 
Since $S^0_i$ is dense in $S_i$ and $O := \lbrace (L'_1, L'_2) \in S_1 \times S_2 \mid \dim (L'_1 + L'_2) = k \rbrace$ is a dense and open subset of $Z := \lbrace (L'_1, L'_2) \in S_1 \times S_2 \mid \dim (L'_1 + L'_2) \leq k \rbrace$, it follows that $U = (S^0_1 \times S^0_2) \cap O$ is dense in $Z$.
Thus, the Zariski closure of the image of the rational map
\begin{align*}
Z &\longdashrightarrow \Gr(k, \PP^n), \\
(L'_1, L'_2) &\longmapsto L'_1 + L'_2,
\end{align*}
which is defined on all of $O$, 
is contained in $\Sec_k(X)$.
\end{proof}

\section{Coisotropic Varieties}
\label{sec:coiso}

In this article, we are mainly interested in Grassmann secant varieties of Segre varieties.
For two vector spaces $U$ and $W$, a key player for our definition of coisotropic varieties is
\begin{align*}
\mathrm{Sec}_k(\Seg(U,W)) = \overline{\left\lbrace L \mid L = \spann(L \cap \Seg(U,W)) \right\rbrace} \subset \Gr(k, \PP(\Hom(U,W))).
\end{align*}

\begin{defn}
\label{defn:coisotropy}
An irreducible subvariety $\Sigma \subset \Gr(\ell, \PP^n)$ of codimension $c \geq 1$ is \emph{coisotropic} if, for every $L \in \Reg(\Sigma)$, 
the conormal space of $\Sigma$ at $L$ is spanned by rank one homomorphisms or is in the Zariski closure of the set of such spaces, i.e.,
$$\PP(N_{\Sigma,L}^\ast) \in \Sec_{c-1}(\Seg(\KK^{n+1}/\aff{L}, \aff{L})).$$
Moreover, $\Sigma$ is \emph{strongly coisotropic} if, for every $L \in \Reg(\Sigma)$, the rank of every homomorphism in the conormal space of $\Sigma$ at $L$ is at most one, i.e., 
$$\PP(N_{\Sigma,L}^\ast) \subset \Seg(\KK^{n+1}/\aff{L}, \aff{L}).$$
\end{defn}

First, we prove Proposition~\ref{prop:keyPropCoiso}. Its two parts follow essentially from Lemmas~\ref{lem:grSecSpans} and~\ref{lem:grSecSum}.
Let us recall that two subvarieties $X_1, X_2 \subset Y$ intersect \emph{transversely at a point $x \in X_1 \cap X_2$} if $X_1$, $X_2$ and $Y$ are all smooth at $x$ and $T_{X_1,x} + T_{X_2,x} = T_{Y,x}$.
The subvarieties $X_1, X_2 \subset Y$ are said to intersect \emph{generically transversely} if
every irreducible component of $X_1 \cap X_2$ contains a point where $X_1$ and $X_2$ are transverse.

\begin{proof}[Proof of Proposition~\ref{prop:keyPropCoiso}]
For the first part, let $\Sigma \subset \Gr(\ell, \PP^n)$ be a subvariety with $\dim \Sigma \leq n-1$.
The codimension $c$ of $\Sigma$ in $\Gr(\ell, \PP^n)$ is at least $(\ell+1)(n-\ell)-(n-1) = \ell(n-\ell-1)+1$. 
Moreover, the codimension of $\Seg(\KK^{n+1}/\aff{L}, \aff{L})$ in $\PP (\Hom(\KK^{n+1}/\aff{L}, \aff{L}))$ is $\ell(n-\ell-1)$ for every $L \in \Gr(\ell, \PP^n)$.
Applying Lemma~\ref{lem:grSecSpans} to $k := c-1 \geq \ell(n-\ell-1)$ and $X := \Seg(\KK^{n+1}/\aff{L}, \aff{L}) \subset \PP (\Hom(\KK^{n+1}/\aff{L}, \aff{L}))$
yields $$\Sec_{c-1}(\Seg(\KK^{n+1}/\aff{L}, \aff{L})) = \Gr(c-1, \PP(\Hom(\KK^{n+1}/\aff{L}, \aff{L}))).$$
In particular, $\Sigma$ is coisotropic.

\smallskip
For the second part, let $\Sigma$ be a non-empty irreducible component of $\Sigma_1 \cap \Sigma_2$, and let $L \in \Sigma$ be a general point.
Since $\Sigma_1$ and $\Sigma_2$ intersect transversely at $L$, we have that $T_{\Sigma,L} = T_{\Sigma_1, L} \cap T_{\Sigma_2,L}$ and $N_{\Sigma,L}^\ast = N_{\Sigma_1,L}^\ast + N_{\Sigma_2,L}^\ast$.
Hence, we can apply Lemma~\ref{lem:grSecSum} to the latter equality:
$$
\PP \left(N_{\Sigma,L}^\ast \right) = \PP \left(N_{\Sigma_1,L}^\ast \right)+ \PP \left(N_{\Sigma_2,L}^\ast \right) \in \Sec_{c-1}(\Seg(\KK^{n+1} / \aff{L}, \aff{L})),
$$
where $c$ is the codimension of $\Sigma$ in $\Gr(\ell, \PP^n)$. 
Since $L \in\Sigma$ was chosen generally, we have shown $\PP \left(N_{\Sigma,L}^\ast \right) \in \Sec_{c-1}(\Seg(\KK^{n+1} / \aff{L}, \aff{L}))$ for every $L \in \Reg(\Sigma)$.
\end{proof}

\subsection{Strongly Coisotropic Varieties}
\label{ssec:stronglyCoisotropic}
In this subsection, we prove Theorem~\ref{thm:stronglyCoisoFULL} and Lemma~\ref{lem:dual}.

\begin{prop}
\label{prop:associatedVarsAreStronglyCoisotropic}
For every irreducible variety $X \subset \PP^n$ and $\ell \leq \codim X-1$, the $\ell$-th higher associated variety $\mathcal{G}_\ell(X) \subset \Gr(\ell, \PP^n)$ of $X$ is strongly coisotropic.
\end{prop}

\begin{proof}
We first observe that $\ell$ is sufficiently small such that every $\ell$-dimensional projective subspace $L \subset \PP^n$ meeting $X$ intersects $X$ non-transversely, i.e.,
$$\mathcal{G}_\ell(X) = \left\lbrace L \in \Gr(\ell, \PP^n) \mid L \cap X \neq \emptyset \right\rbrace.$$
Hence, a general $L \in \mathcal{G}_\ell(X)$ intersects $X$ at exactly one point $x_L$, which gives us a rational map
\mbox{$\varpi: \mathcal{G}_\ell(X) \dashrightarrow X$}, $L \mapsto x_L$.
By Corollary~\ref{cor:keyLemmaTangentCorrespondence}, we have $\varphi |_{\aff{x_L}} = (D_L \varpi (\varphi) \mod \aff{L})$ for a general $L \in \mathcal{G}_\ell(X)$ and every $\varphi \in T_{\mathcal{G}_\ell(X), L}$.
Since the image of $D_L \varpi (\varphi)$ is contained in $\aff{\TT_{X,x_L}} / \aff{x_L}$, the tangent space 
$T_{\mathcal{G}_\ell(X), L}$ is contained in $$V_L := \lbrace \varphi \in \Hom(\aff{L}, \KK^{n+1} / \aff{L}) \mid \varphi(\aff{x_L}) \subset (\aff{\TT_{X,x_L}}+\aff{L})/\aff{L} \rbrace.$$
Since $L$ was chosen generally, the dimension of $\TT_{X,x_L} + L$ is $\dim X + \ell$, which shows that $V_L$ has dimension $\ell(n-\ell)+\dim X$.
This is also the dimension of the variety $\mathcal{G}_\ell(X)$.
Hence, we have derived $T_{\mathcal{G}_\ell(X), L} = V_L$.
A similar dimension count yields
\begin{small}
\begin{align*}
%\label{eq:tangentGeneralChow}
V_L =
%\left\lbrace \varphi \in \Hom(\aff{L}, V/\aff{L}) \mid \varphi(\aff{x_L}) \subset (\aff{\TT_{X,x_L}}+\aff{L}) / \aff{L} \right\rbrace \\
%&= \left\lbrace \varphi \in \Hom(\aff{L}, V/\aff{L}) \mid \aff{x_L} \subset \ker \varphi \right\rbrace
%+ \left\lbrace \varphi \in \Hom(\aff{L}, V/\aff{L}) \mid \im \varphi \subset (\aff{\TT_{X,x_L}}+\aff{L}) / \aff{L} \right\rbrace,
 \left\lbrace \varphi \in T_{\Gr(\ell, \PP^n),L} \mid \aff{x_L} \subset \ker \varphi \right\rbrace
+ \left\lbrace \varphi \in T_{\Gr(\ell, \PP^n),L} \mid \im \varphi \subset (\aff{\TT_{X,x_L}}+\aff{L}) / \aff{L} \right\rbrace.
\end{align*}
\end{small}
This shows that the conormal space of $\mathcal{G}_\ell(X)$ at a general point $L \in \mathcal{G}_\ell(X)$ contains only homomorphisms of rank one:
\begin{align}
\label{eq:conormalSpaceG_l(X)}
\hspace*{-3mm}
N^\ast_{\mathcal{G}_\ell(X), L} = \left\lbrace \varphi \in \Hom(\KK^{n+1}/\aff{L}, \aff{L}) \mid  (\aff{\TT_{X,x_L}}+\aff{L}) / \aff{L} \subset \ker  \varphi, \, \im \varphi \subset \aff{x_L} \right\rbrace. 
\hspace*{-2mm}
\end{align}
Since $L$ was chosen generally, we have shown $\PP(N^\ast_{\mathcal{G}_\ell(X), L}) \subset \Seg(\KK^{n+1}/\aff{L},\aff{L})$ for every $L \in \Reg(\mathcal{G}_\ell(X))$.
\end{proof}

We will make use of the following identification of the higher associated varieties of a projective variety $X$ and of its dual variety $X^\vee$.
Note that the following lemma, together with Theorem~\ref{thm:stronglyCoisoFULL}, yields Lemma~\ref{lem:dual}.

\begin{lem}
\label{lem:extendedDuality}
For an irreducible variety $X \subset \PP^n$ and $\ell \in \lbrace 0, 1, \ldots, n-1 \rbrace$, we have
$$ \mathcal{G}_\ell(X)^\perp = \mathcal{G}_{n-\ell-1} (X^\vee). $$
\end{lem}

\begin{proof}
For points $x \in \PP^n$ and $y \in (\PP^n)^\ast$ such that $x \in y^\vee$, we define $$\mathcal{G}_\ell(x,y) := \lbrace L \in \Gr(\ell, \PP^n) \mid x \in L \subset y^\vee \rbrace.$$
We clearly have that $\mathcal{G}_\ell(x,y)^\perp = \mathcal{G}_{n-\ell-1}(y,x)$.
Moreover, we consider the following open subset of the conormal variety of $X$:
$$\mathcal{U}_{X, X^\vee}  :=  \left\lbrace (x,y) \in \Reg(X) \times \Reg(X^\vee) \mid \TT_{X,x} \subset y^\vee \right\rbrace.$$
The biduality theorem implies the equality $\mathcal{U}_{X^\vee, X} = \lbrace (y,x) \mid (x,y) \in \mathcal{U}_{X, X^\vee} \rbrace$.
Since the condition $\dim(L + \TT_{X,x})< n$ in the definition of $\mathcal{G}_\ell(X)$ means that $L$ is contained in a tangent hyperplane of $X$ at $x$, we have $\mathcal{G}_\ell(X) = \overline{\bigcup_{(x,y) \in \mathcal{U}_{X,X^\vee}}\mathcal{G}_\ell(x,y)}$.
Hence, 
\begin{align*}
\mathcal{G}_\ell(X)^\perp = \overline{\bigcup_{(x,y) \in \mathcal{U}_{X,X^\vee}}\mathcal{G}_\ell(x,y)^\perp}
= \overline{\bigcup_{(y,x) \in \mathcal{U}_{X^\vee,X}}\mathcal{G}_{n-\ell-1}(y,x)} 
= \mathcal{G}_{n-\ell-1}(X^\vee).
\end{align*}
\end{proof}

\begin{cor}
\label{cor:stronglyCoisotropicVarietiesAssociatedToX}
The $\ell$-th higher associated variety $\mathcal{G}_\ell(X)$ of an irreducible variety $X \subset \PP^n$ is strongly coisotropic for each $\ell \in \lbrace 0, 1, \ldots, n-1 \rbrace$.
\end{cor}

\begin{proof}
If $\codim X -1 \leq \ell \leq \dim X^\vee$, then $\mathcal{G}_\ell(X)$ is a hypersurface in its ambient Grassmannian $\Gr(\ell, \PP^n)$ by Theorem~\ref{thm:coisotropicHSdimension}.
Hence, it is coisotropic by~\cite[Ch.~4, Thn.~3.14]{gkz} or~\cite{coisotropicHS}.
If $\ell < \codim X -1$, then $\mathcal{G}_\ell(X)$ is strongly coisotropic by Proposition~\ref{prop:associatedVarsAreStronglyCoisotropic}.
Finally, if $\ell > \dim X^\vee$, then we consider $\mathcal{G}_\ell(X)^\perp = \mathcal{G}_{n-\ell-1} (X^\vee)$ (by Lemma~\ref{lem:extendedDuality}) and observe that $n-\ell-1 < \codim X^\vee -1$.
Thus, $\mathcal{G}_{n-\ell-1} (X^\vee)$ is strongly coisotropic by Proposition~\ref{prop:associatedVarsAreStronglyCoisotropic}, which shows that $\mathcal{G}_\ell(X)$ is also strongly coisotropic by Remark/Definition~\ref{rem:dualityConormalSpaces}.
\end{proof}

In fact, every strongly coisotropic variety is associated to a projective variety as in Corollary~\ref{cor:stronglyCoisotropicVarietiesAssociatedToX}.
We prove this assertion in the remainder of this subsection. 

\begin{thm}
\label{thm:stronglyCoisotropic}
For an irreducible strongly coisotropic variety $\Sigma \subset \Gr(\ell, \PP^n)$, there is an irreducible variety $X \subset \PP^n$ such that $\Sigma = \mathcal{G}_\ell(X)$.
\end{thm}

Since this statement is already shown for coisotropic hypersurfaces in~\cite[Ch.~4, Thm.~3.14]{gkz} or~\cite{coisotropicHS}, we consider only strongly coisotropic varieties of codimension at least two in the following.
For this, we use the notion of $\alpha$- and $\beta$-spaces introduced in Definition~\ref{def:alphaBetaSpace}.

\begin{lem}
\label{lem:alphaBetaSegre}
Let $U$ and $W$ be vector spaces.
Every linear space $E$ of dimension at least two which is contained in the affine cone $\lbrace \varphi \in \Hom(U,W) \mid \rank (\varphi) \leq 1 \rbrace$ over the Segre variety $\Seg(U,W)$ is contained in a unique $\alpha$- or $\beta$-space.
\end{lem}

\begin{proof}
Since the dimension of the intersection of two distinct $\alpha$-spaces is zero (similarly for $\beta$-spaces) and the intersection of an $\alpha$- with a $\beta$-space is one-dimensional, it is enough to show that $E$ is contained in some $\alpha$- or $\beta$-space.

If $E$ would neither be contained in an $\alpha$- nor in a $\beta$-space, then $E$ would contain $\varphi_1, \varphi_2, \psi_1, \psi_2 \in \Hom(U,W) \setminus \lbrace 0 \rbrace$ with $\ker \varphi_1 \neq \ker \varphi_2$ and $\im \psi_1 \neq \im \psi_2$.
Among these four homomorphisms are two with both, distinct kernels \emph{and}  distinct images, say $\chi_1$ and $\chi_2$.
Let us pick $u_1 \in \ker \chi_1 \setminus \ker \chi_2$ as well as  $u_2 \in \ker \chi_2 \setminus \ker \chi_1$.
We see that $(\chi_1 + \chi_2)(u_1) \in \im \chi_2 \setminus \lbrace 0 \rbrace$ and analogously $(\chi_1 + \chi_2)(u_2) \in \im \chi_1 \setminus \lbrace 0 \rbrace$.
Thus, the image of $\chi_1 + \chi_2 \in E$ must be at least two-dimensional, which contradicts that all homomorphisms in $E$ have rank at most one.
\end{proof}

\begin{defCor}
Let $\Sigma$ be an irreducible strongly coisotropic variety of codimension at least two.
Either each conormal space at a smooth point of $\Sigma$ is contained in a unique $\alpha$-space, or each conormal space at a smooth point of $\Sigma$ is contained in a unique $\beta$-space.
In the first case, we call $\Sigma$ \emph{strongly coisotropic of $\alpha$-type}. 
In the latter case, we say that $\Sigma$ is \emph{strongly coisotropic of $\beta$-type}. 
\end{defCor}

\begin{cor}
\label{cor:dimStronglyCoiso}
Every strongly coisotropic variety in $\Gr(\ell, \PP^n)$ of $\alpha$-type has codimension at most $\ell+1$, 
and every strongly coisotropic variety of $\beta$-type has codimension at most $n-\ell$.
\end{cor}

\begin{proof}
This follows from the fact that $\alpha$-spaces in $\Hom(\KK^{n+1}/\aff{L}, \aff{L})$ have dimension $\ell+1$ and that the dimension of $\beta$-spaces is $n-\ell$, where $L \in \Gr(\ell, \PP^n)$. 
\end{proof}

\begin{lem}
\label{lem:alphaBetaDualityCoiso}
A subvariety $\Sigma \subset \Gr(\ell,\PP^n)$ is strongly coisotropic of $\beta$-type if and only if $\Sigma^\perp \subset \Gr(n-\ell-1, (\PP^n)^\ast)$ is strongly coisotropic of $\alpha$-type.
\end{lem}

\begin{proof}
First, we notice that $L$ is a smooth point of $\Sigma$ if and only if $L^\vee$ is a smooth point of $\Sigma^\perp$.
Secondly, the image of $\varphi \in N^\ast_{\Sigma,L}$ is contained in a linear subspace $U \subset \aff{L}$ if and only the kernel of $\varphi^\ast \in N^\ast_{\Sigma^\perp, L^\vee}$ contains $(\aff{L}/U)^\ast$.
\end{proof}

\begin{cor}
Let $X \subset \PP^n$ be an irreducible variety.
Table~\ref{tab:typesCH_i(X)} summarizes the types of the strongly coisotropic varieties associated to $X$ (see~\eqref{eq:conormalSpaceG_l(X)} ).
$\hfill\diamondsuit$
\addtolength{\tabcolsep}{-1pt} 
\begin{table}
\centering
\begin{tabular}{cc}
\Xhline{2\arrayrulewidth}
$\ell$ & type of $\mathcal{G}_\ell(X)$ \\
\hline
$0$ \\
\vdots & $\beta$-type \\
$\codim X -2$ \\
\hline
$\codim X -1$ \\
 \vdots & hypersurface  \\
 $\dim X^\vee$ \\
\hline
 $\dim X^\vee+1$ \\
 \vdots & $\alpha$-type \\
 $n-1$ \\
\Xhline{2\arrayrulewidth}
\end{tabular}
\caption{Types of strongly coisotropic varieties of an irreducible projective variety $X$.}
\label{tab:typesCH_i(X)}
\end{table}
\addtolength{\tabcolsep}{1pt} 
\end{cor}

\begin{example}
\label{ex:segreContinued}
For the Segre variety $X := \Seg(\KK^2, \KK^4) \subset \PP^7$ in Example~\ref{ex:segreStart}, 
the two strongly coisotropic varieties $\mathcal{G}_0(X) = X$ and $\mathcal{G}_1(X)$ are of $\beta$-type,
the three varieties $\mathcal{G}_2(X), \ldots, \mathcal{G}_4(X)$ are its coisotropic hypersurfaces,
and the two strongly coisotropic varieties $\mathcal{G}_5(X)$ and $\mathcal{G}_6(X) = X^\vee$ are of $\alpha$-type.
$\hfill\diamondsuit$
\end{example}

Now we will finally prove Theorem~\ref{thm:stronglyCoisotropic} (and thus Theorem~\ref{thm:stronglyCoisoFULL}) by restricting ourselves to strongly coisotropic varieties of $\beta$-type.

\begin{thm}
\label{thm:stronglyCoisotropicBeta}
For an irreducible strongly coisotropic variety $\Sigma \subset \Gr(\ell, \PP^n)$ of $\beta$-type, there is an irreducible variety $X \subset \PP^n$ such that $\Sigma = \mathcal{G}_{\ell}(X)$.
\end{thm}

\begin{proof}
Since $\Sigma$ is strongly coisotropic of $\beta$-type, there is, for every $L \in \Reg(\Sigma)$, a unique $x_L \in L$ such that
the conormal space $N^\ast_{\Sigma, L}$ is contained in the $\beta$-space $E_\beta(\aff{x_L})$.
In other words, for every $L \in \Reg(\Sigma)$ and every $\varphi \in N^\ast_{\Sigma, L}$, the image of $\varphi$ is contained in the one-dimensional subspace $\aff{x_L} \subset \KK^{n+1}$ corresponding to the projective point $x_L$.
Hence, we get a rational map $\rho: \Sigma \dashrightarrow \PP^n$ which maps $L \in \Reg(\Sigma)$ to $x_L \in \PP^n$.
We denote the Zariski closure of the image of this map by $X \subset \PP^n$. 
Note that $X$ is irreducible since $\Sigma$ is irreducible.

Moreover, for every $L \in \Reg(\Sigma)$, we consider the intersection of the kernels of all $\varphi \in N^\ast_{\Sigma, L}$ to find $P_L \in \Gr(n- \codim \Sigma, \PP^n)$ with $L \subset P_L$ and $\aff{P_L}/\aff{L} \subset \ker \varphi$ for every $\varphi \in N^\ast_{\Sigma, L}$.
Thus, for $L \in \Reg(\Sigma)$, we see that
\begin{align}
\label{eq:conormalSpaceBeta}
N^\ast_{\Sigma,L} = \left\lbrace \varphi \in \Hom(\KK^{n+1}/\aff{L},  \aff{L}) \mid \im \varphi \subset \aff{x_L}, \, \aff{P_L}/\aff{L} \subset \ker \varphi \right\rbrace,
\end{align}
since both spaces in~\eqref{eq:conormalSpaceBeta} have the same dimension (namely $\codim \Sigma$).
This shows
\begin{small}
\begin{align*}
T_{\Sigma,L} 
&= \lbrace \varphi \in \Hom(\aff{L}, \KK^{n+1}/\aff{L}) \,|\, \aff{x_L} \subset \ker \varphi \rbrace + \lbrace \varphi \in \Hom(\aff{L}, \KK^{n+1}/\aff{L}) \,|\, \im \varphi \subset \aff{P_L}/\aff{L} \rbrace
\\ &= \lbrace \varphi \in \Hom(\aff{L}, \KK^{n+1}/\aff{L}) \,|\, \varphi(\aff{x_L}) \subset \aff{P_L}/\aff{L} \rbrace.
\end{align*}
\end{small}
The differential of $\rho: \Sigma \dashrightarrow X$ at a general $L \in \Sigma$ is a surjection $T_{\Sigma, L} \twoheadrightarrow T_{X, x_L}$.
By Corollary~\ref{cor:keyLemmaTangentCorrespondence}, the image of each $\psi:\aff{x_L} \to \KK^{n+1}/\aff{x_L}$ in $T_{X, x_L}$ is contained in $\aff{P_L}/\aff{x_L}$.
Thus, $\TT_{X, x_L} \subset P_L$ for a general $L \in \Sigma$.
In particular, the dimension of $X$ is at most $\dim P_L = n- \codim \Sigma$. 
Since we have constructed $X$ such that $\Sigma \subset \mathcal{G}_\ell(X)$, we derive
\begin{align}
\label{eq:estimatestronglyCoiso}
\hspace*{-3mm}
\begin{split}
\dim \Sigma &\leq \dim \mathcal{G}_\ell(X) \leq \ell(n-\ell) + \dim X \\
&\leq \ell(n-\ell) + n- \codim \Sigma
= \ell(n-\ell) + n - (\ell+1)(n-\ell)+\dim \Sigma\\
&= \dim \Sigma + \ell.
\end{split}
\end{align}
This shows that $\dim X \geq \dim \Sigma - \ell(n-\ell) = n-\codim \Sigma-\ell$.
If $\dim X = n-\codim \Sigma-\ell$, then all inequalities in the first two of~\eqref{eq:estimatestronglyCoiso} are equalities and $\Sigma = \mathcal{G}_\ell(X)$. 
Hence, it is only left to show that the dimension of $X$ cannot be larger than $n-\codim \Sigma-\ell$.

Let us first assume that $\dim X = n-\codim \Sigma > n-\codim \Sigma-\ell$.
In this case, we have $\ell > 0$, $P_L = \TT_{X, x_L}$ for a general $L \in \Sigma$, and 
\begin{align*}
\Sigma \subset \Sigma' := \overline{\left\lbrace L \mid \exists x \in \Reg(X): x \in L \subset \TT_{X, x_L} \right\rbrace} \subset \Gr(\ell, \PP^n).
\end{align*}
This yields 
\begin{align*}
\dim \Sigma
&\leq \dim \Sigma'
\leq \dim X + \dim \Gr(\ell-1, \PP^{\dim X-1}) \\
&= (n- \codim \Sigma) + \ell(n-\codim \Sigma - \ell)\\
&= \dim \Sigma - \ell(n-\ell-1) + \ell \left( \dim \Sigma - \ell(n-\ell) \right), 
\end{align*}
which is, due to $\ell > 0$, equivalent to $\dim \Sigma \geq (\ell+1)(n-\ell)-1$.
The latter inequality is a contradiction to $\codim \Sigma \geq 2$.

Finally, we assume that $n -\codim \Sigma-\ell < \dim X < n - \codim \Sigma$. 
We define $\delta := \dim X - (n -\codim \Sigma-\ell) = \dim X - \dim \Sigma + \ell(n-\ell)> 0$.
For a general $L \in \Sigma$, we have $\TT_{X,x_L}+L \subset P_L$ and thus $\dim (\TT_{X,x_L} \cap L) \geq \delta$.
This shows that
\begin{align*}
\Sigma \subset \Sigma'' := \overline{\left\lbrace L \mid \exists x \in \Reg(X): x \in L, \dim (\TT_{X,x} \cap L) \geq \delta \right\rbrace} \subset \Gr(\ell, \PP^n).
\end{align*}
We can parametrize a general point in $\Sigma''$ by a point $x \in \Reg(X)$, a $\delta$-di\-men\-sion\-al subspace $L'$ of $\TT_{X,x}$ passing through $x$, and an $\ell$-dimensional subspace of $\PP^n$ containing $L'$. This leads us to
\begin{align*}
\dim \Sigma &\leq
\dim \Sigma'' \\
&\leq \dim X + \dim \Gr(\delta-1, \PP^{\dim X-1}) + \dim \Gr (\ell-\delta-1, \PP^{n-\delta-1}) \\
&= \dim X + \delta(\dim X-\delta)+(\ell-\delta)(n-\ell) =: D.
\end{align*}
Note that $\ell-\delta-1 \geq 0$ due to $\dim X < n - \codim \Sigma$.
By definition of $D$ and $\delta$, 
\begin{align*}
0 \leq D - \dim \Sigma
= \delta \left( 1 + (\dim X - \delta)  - (n-\ell) \right) 
= - \delta \left( \codim \Sigma -1  \right).
\end{align*}
This is a contradiction since $\codim \Sigma \geq 2$ and $\delta > 0$.
\end{proof}

\begin{proof}[Proof of Theorems~\ref{thm:stronglyCoisoFULL} and~\ref{thm:stronglyCoisotropic}.]
Theorem~\ref{thm:stronglyCoisoFULL} is an amalgamation of Corollary~\ref{cor:stronglyCoisotropicVarietiesAssociatedToX} and Theorem~\ref{thm:stronglyCoisotropic}. Hence, it is sufficient to prove the latter:

If $\Sigma$ is a hypersurface, the assertion is proved in~\cite[Ch.~4, Thm.~3.14]{gkz} and~\cite{coisotropicHS}.
If $\Sigma$ is strongly coisotropic of $\beta$-type, the assertion follows from Theorem~\ref{thm:stronglyCoisotropicBeta}.
Finally, if $\Sigma$ is strongly coisotropic of $\alpha$-type, then $\Sigma^\perp$ is strongly coisotropic of $\beta$-type by Lemma~\ref{lem:alphaBetaDualityCoiso}.
Thus, Theorem~\ref{thm:stronglyCoisotropicBeta} implies that $\Sigma^\perp = \mathcal{G}_{n-\ell-1}(X)$ for some irreducible variety $X \subset (\PP^n)^\ast$.
By Lemma~\ref{lem:extendedDuality}, we get $\Sigma = \mathcal{G}_{\ell}(X^\vee)$.
\end{proof}

\subsection{Lines with Higher Contact to Hypersurfaces}
\label{ssec:higherContact}
For a hypersurface $X \subset \PP^n$ and $1 \leq m \leq \deg X$, we define
\begin{align*}
\mathcal{L}_m(X) &:= \overline{\left\lbrace L \mid \exists p \in \Reg(X): L \text{ intersects } X \text{ at } p \text{ with multiplicity } m \right\rbrace} \\
&\; \subset \Gr(1, \PP^n).
\end{align*}
Throughout this subsection, we consider a general hypersurface $X \subset \PP^n$ of degree at least three and assume $m \leq n$.
The subvariety $\mathcal{L}_m(X)$ of the Grassmannian of lines in $\PP^n$ has codimension $m-1$ (cf. Lemma~\ref{lem:contactOrderCone}).
We will show that $\mathcal{L}_m(X)$ is coisotropic, but that its conormal spaces are \emph{not} spanned by rank one homomorphisms if $m \geq 3$:

\begin{thm}
\label{thm:higherContactLines}
%For $L \in \Reg(\mathcal{L}_m(X)) \setminus \mathcal{L}_{m+1}(X)$ which intersects $X$ at exactly one point $p \in \Reg(X)$ with multiplicity $m$, we have that
If $m \geq 2$, a general $L \in \mathcal{L}_m(X)$ intersects $X$ at exactly one point $p_L \in \Reg(X)$ with multiplicity $m$ and
$$\PP(N^\ast_{\mathcal{L}_m(X),L}) \cap \Seg(\KK^{n+1}/\aff{L}, \aff{L}) = \lbrace \varphi \rbrace,$$
where $\im \varphi = \aff{p_L}$ and $\ker \varphi = \aff{\TT_{X,p_L}} / \aff{L}$.
Moreover, the $(m-2)$-dimensional projectivized conormal space $\PP(N^\ast_{\mathcal{L}_m(X),L})$ intersects $\Seg(\KK^{n+1}/\aff{L}, \aff{L})$ at $\varphi$ with multiplicity $m-1$.
In particular, $$\PP(N^\ast_{\mathcal{L}_m(X),L}) \in \Sec_{m-2} (\Seg(\KK^{n+1}/\aff{L}, \aff{L}))$$
and $\mathcal{L}_m(X)$ is coisotropic. 
\end{thm}

Note that Theorem~\ref{thm:higherContactLines} is simply a more specific reformulation of Theorem~\ref{thm:higherContactLinesFULL}.

\begin{example}
\label{ex:conormalSpaceHurwitz}
$\mathcal{L}_2(X) = \mathcal{G}_1(X)$ is the Hurwitz hypersurface of $X$.
For a general $L \in \mathcal{L}_2(X)$, we have that $\PP(N^\ast_{\mathcal{L}_2(X),L})$ consists of one projectivized homomorphism with image $\aff{p_L}$ and kernel $\aff{\TT_{X,p_L}} / \aff{L}$.

The projectivized conormal space at a general $L \in \mathcal{L}_3(X)$ is a tangent line to the Segre variety $\Seg(\KK^{n+1} / \aff{L}, \, \aff{L})$ at this homomorphism.
$\hfill\diamondsuit$
\end{example}

To prove Theorem~\ref{thm:higherContactLines}, we further define
\begin{align*}
\mathcal{L}_{m,p}(X) := \overline{\left\lbrace L \mid L \text{ intersects } X \text{ at } p \text{ with multiplicity } m \right\rbrace} \subset \Gr(1, \PP^n)
\end{align*}
for any $p \in \Reg(X)$ and denote by $\mathcal{C}_{m,p}(X) \subset \PP^n$ the union of all $L \in \mathcal{L}_{m,p}(X)$.

\begin{lem}
\label{lem:contactOrderCone}
For a general $p \in X$, the cone $\mathcal{C}_{m,p} (X)$ has codimension $m-1$ in $\PP^n$ and degree $(m-1)!$. Moreover, it is smooth everywhere, except at $p$ if $m \geq 3$.
\end{lem}

\begin{proof}
We may assume that $p$ is the origin of an affine chart of $\PP^n$.
We consider the defining equation $f(x) = f_1(x) + f_2(x) + \ldots$ of $X$ in this affine chart, where $f_i$ is a homogeneous polynomial of degree $i$.
In this affine chart, the cone $\mathcal{C}_{m,p}(X)$ is the zero locus of $\lbrace f_1(x), \ldots, f_{m-1}(x)  \rbrace$.
Since $X$ and $p$ are general, the polynomials $f_1, \ldots, f_{m-1}$ define smooth irreducible hypersurfaces whose projectivizations intersect transversely.
\end{proof}

Note that $\mathcal{C}_{1,p}(X) = \PP^n$ and that $\mathcal{C}_{2,p}(X) = \TT_{X,p}$. % is the tangent hyperplane to $X$ at $p$.
Each cone $\mathcal{C}_{m,p}(X)$ is a hypersurface in the cone $\mathcal{C}_{m-1,p}(X)$.
In particular, for a line $L \in \mathcal{L}_{m, p}(X)$, we have a flag of projective spaces 
\begin{align}
\label{eq:higherContactFlag}
p \in L \subset \TT_L \mathcal{C}_{m, p}(X) \subset \TT_L\mathcal{C}_{m-1, p}(X) \subset \ldots \subset \TT_L \mathcal{C}_{2, p}(X) = \TT_{X,p} \subset \PP^n,
\end{align}
where $\TT_L \mathcal{C}_{k, p}(X)$ is a hyperplane in $\TT_L \mathcal{C}_{k-1, p}(X)$ for $2 < k \leq m$.
Here $\TT_L \mathcal{C}_{k, p}(X)$ denotes the unique embedded tangent space of the cone $\mathcal{C}_{k, p}(X)$ along $L$.
Besides, we denote by $T_L \mathcal{C}_{k, p}(X) := \lbrace \varphi: \aff{p} \to \KK^{n+1}/\aff{p} \mid \im \varphi \subset \aff{\TT_L \mathcal{C}_{k, p}(X)} / \aff{p} \rbrace$ the corresponding subspace of $T_{\PP^n,p}$.

For $m \geq 2$, we consider the rational map $\mathcal{L}_m(X) \dashrightarrow X$ which sends a general $L \in \mathcal{L}_m(X)$ to the unique point $p_L \in \Reg(X)$ at which $L$ intersects $X$ with multiplicity $m$.
For a general $L \in \mathcal{L}_m(X)$, the differential of this map at $L$ is a surjection $\Phi_L: T_{\mathcal{L}_m(X),L} \twoheadrightarrow T_{X, p_L}$. 
We can prove Theorem~\ref{thm:higherContactLines} by computing the preimages of $\Phi_L$ along the flag in~\eqref{eq:higherContactFlag}.

\begin{lem}
\label{lem:higherContactPhi}
For $m \geq 2$ and $L \in \mathcal{L}_m(X)$ general, we have
\begin{align}
\label{eq:HCker}
\ker \Phi_L &= \lbrace \varphi \in \mathcal{H} \mid \aff{p_L} \subset \ker \varphi, \, \im \varphi \subset \aff{\TT_L \mathcal{C}_{m, p_L}(X)}/\aff{L} \rbrace, \\
\label{eq:HCL}
\Phi_L^{-1}(T_{L,p_L}) &= \lbrace \varphi \in \mathcal{H} \mid \aff{p_L} \subset \ker \varphi, \, \im \varphi \subset \aff{\TT_L \mathcal{C}_{m-1, p_L}(X)}/\aff{L} \rbrace, \\
\label{eq:HCC_m}
\Phi_L^{-1}(T_L \mathcal{C}_{m, p_L}(X)) &= \left\lbrace \varphi \in \mathcal{H} \;\middle\vert \begin{array}{rcl} \varphi(\aff{p_L}) &\subset& \aff{\TT_L \mathcal{C}_{m, p_L}(X)}/\aff{L}, \\ \im \varphi &\subset& \aff{\TT_L \mathcal{C}_{m-1, p_L}(X)}/\aff{L} \end{array} \right\rbrace, \\
\label{eq:HCC_m-l1}
\Phi_L^{-1}(T_L \mathcal{C}_{m-l, p_L}(X)) &\subset \left\lbrace \varphi \in \mathcal{H} \;\middle\vert \begin{array}{rcl} \varphi(\aff{p_L}) &\subset& \aff{\TT_L \mathcal{C}_{m-l, p_L}(X)}/\aff{L}, \\ \im \varphi &\subset& \aff{\TT_L \mathcal{C}_{m-l-1, p_L}(X)}/\aff{L} \end{array} \right\rbrace, \\
\label{eq:HCC_m-l2}
\Phi_L^{-1}(T_L \mathcal{C}_{m-l, p_L}(X)) &\not\subset \left\lbrace \varphi \in \mathcal{H} \;\middle\vert \begin{array}{rcl} \varphi(\aff{p_L}) &\subset& \aff{\TT_L \mathcal{C}_{m-l+1, p_L}(X)}/\aff{L}, \\ \im \varphi &\subset& \aff{\TT_L \mathcal{C}_{m-l-1, p_L}(X)}/\aff{L} \end{array} \right\rbrace, \\
\label{eq:HCC_m-l3}
\Phi_L^{-1}(T_L \mathcal{C}_{m-l, p_L}(X)) &\not\subset \left\lbrace \varphi \in \mathcal{H} \;\middle\vert \begin{array}{rcl} \varphi(\aff{p_L}) &\subset& \aff{\TT_L \mathcal{C}_{m-l, p_L}(X)}/\aff{L}, \\ \im \varphi &\subset& \aff{\TT_L \mathcal{C}_{m-l, p_L}(X)}/\aff{L} \end{array} \right\rbrace,
\end{align}
where $1 \leq l \leq m-2$ and $\mathcal{H} := \Hom(\aff{L}, \KK^{n+1}/\aff{L})$.
\end{lem}

We prove Theorem~\ref{thm:higherContactLines} before we give the technical proof of Lemma~\ref{lem:higherContactPhi}.
The ideas for the computations in the next two proofs were developed together with Emre Sert\"oz.

\begin{proof}[Proof of Theorem~\ref{thm:higherContactLinesFULL} and Theorem~\ref{thm:higherContactLines}.]
We choose coordinates for $\aff{L}$ and $\KK^{n+1}/\aff{L}$ such that we can write every $\varphi \in \Hom(\aff{L}, \KK^{n+1}/\aff{L})$ as a matrix in $\KK^{2 \times (n-1)}$.
For $\aff{L}$, we pick a basis $\lbrace e_0, e_1\rbrace$ such that $e_0 \in \aff{p_L} \setminus \lbrace 0 \rbrace$ and $e_1 \in \aff{L} \setminus \aff{p_L}$.
For $\KK^{n+1}/\aff{L}$, we pick a basis $\lbrace e_2, \ldots, e_{n} \rbrace$ such that $\aff{\TT_L \mathcal{C}_{k,p_L}(X)}/\aff{L}$ is spanned by $e_2, \ldots, e_{n-k+1}$ for $2 \leq k \leq m$.
The first row of a matrix in $\KK^{2 \times (n-1)}$ encoding a map in $\Hom(\aff{L}, \KK^{n+1}/\aff{L})$ corresponds to $e_0$, the second row to $e_1$, and the $i$-th column corresponds to $e_{i+1}$. 
According to Lemma~\ref{lem:higherContactPhi}, the kernel of $\Phi_L$ is spanned by
\begin{align*}
\begin{blockarray}{cccc}
 &  &  & \\ 
 &  &  & \\  
\begin{block}{[cccc]}
  0 & 0 & \cdots & 0 \\
  1 & 0 & \cdots & 0 \\
\end{block}
\end{blockarray} \, , \quad \ldots \, , \quad
\begin{blockarray}{ccccccc}
 &  &  & \scriptstyle n-m \\ 
 &  &  &  \downarrow \\  
\begin{block}{[ccccccc]}
  0 & \cdots & 0 & 0 & 0 & \cdots & 0 \\
  0 & \cdots & 0 & 1 & 0 & \cdots & 0 \\
\end{block}
\end{blockarray} \, .
\end{align*}
$\Phi_L^{-1}(T_{L,p_L})$ is additionally spanned by 
$ \quad
\begin{blockarray}{ccccccc}
 &  &  & \scriptstyle  n-m+1 \\ 
 &  &  &  \downarrow \\  
\begin{block}{[ccccccc]}
  0 & \cdots & 0 & 0 & 0 & \cdots & 0 \\
  0 & \cdots & 0 & 1 & 0 & \cdots & 0 \\
\end{block}
\end{blockarray}  \, .
$ \\
$\Phi_L^{-1}(T_L \mathcal{C}_{m,p_L}(X))$ is spanned by all the matrices above as well as
\begin{align*}
\begin{blockarray}{cccc}
 &  &  & \\ 
 &  &  & \\  
\begin{block}{[cccc]}
  1 & 0 & \cdots & 0 \\
  0 & 0 & \cdots & 0 \\
\end{block}
\end{blockarray} \, , \quad \ldots \, , \quad
\begin{blockarray}{ccccccc}
 &  &  &\scriptstyle n-m \\ 
 &  &  & \downarrow \\  
\begin{block}{[ccccccc]}
  0 & \cdots & 0 & 1 & 0 & \cdots & 0 \\
  0 & \cdots & 0 & 0 & 0 & \cdots & 0 \\
\end{block}
\end{blockarray} \,.
\end{align*}
$\Phi_L^{-1}(T_L \mathcal{C}_{m-1,p_L}(X))$ is additionally spanned by 
\begin{align*}
\begin{blockarray}{cccccccc}
&  &  &\scriptstyle n-m+1 \\ 
&  &  & \downarrow \\  
\begin{block}{[cccccccc]}
0 & \cdots & 0 & 1 & 0 & 0 & \cdots & 0 \\
0 & \cdots & 0 & 0 & C_1 & 0 & \cdots & 0 \\
\end{block}
\end{blockarray} \, 
\end{align*}
for some non-zero constant $C_1$.
Analogously, $\Phi_L^{-1}(T_L \mathcal{C}_{m-2,p_L}(X))$ is additionally spanned by 
\begin{align*}
\begin{blockarray}{cccccccc}
&  &  &\scriptstyle n-m+2 \\ 
&  &  & \downarrow \\  
\begin{block}{[cccccccc]}
0 & \cdots & 0 & 1 & 0 & 0 & \cdots & 0 \\
0 & \cdots & 0 & c_{2,1} & C_2 & 0 & \cdots & 0 \\
\end{block}
\end{blockarray} \, , 
\end{align*}
where $C_2 \neq 0$ and $c_{2,1}$ are constants.
More generally, for $1 \leq l \leq m-2$, a non-zero matrix in $\Phi_L^{-1}(T_L \mathcal{C}_{m-l,p_L}(X)) \setminus \Phi_L^{-1}(T_L \mathcal{C}_{m-l+1,p_L}(X))$ is
\begin{align*}
\begin{blockarray}{ccccccccccc}
&  &  & & & & \scriptstyle n-m+l \\ 
&  &  & & & & \downarrow \\  
\begin{block}{[ccccccccccc]}
0 & \cdots & 0 & 0       & \cdots & 0         & 1 & 0 & 0 & \cdots & 0 \\
0 & \cdots & 0 & c_{l,1} & \cdots & c_{l,l-2} & c_{l,l-1} & C_l & 0 & \cdots & 0 \\
\end{block}
\end{blockarray} \, , 
\end{align*}
where $C_l \neq 0$ and $c_{l,1}, \ldots, c_{l,l-1}$ are some constants.
In particular, for $l=m-2$, this matrix is
\begin{align*}
\begin{blockarray}{cccccccc}  
\begin{block}{[cccccccc]}
  0 & \cdots & 0 & 0         & \cdots & 0       & 1       & 0 \\
  0 & \cdots & 0 & c_{m-2,1} & \cdots & c_{m-2,m-4} & c_{m-2,m-3} & C_{m-2} \\
\end{block}
\end{blockarray} \, .
\end{align*}
All these matrices together span $T_{\mathcal{L}_m(X),L}$. 
Thus, the conormal space $N^\ast_{\mathcal{L}_m(X),L}$ is spanned by maps in $\Hom(\KK^{n+1}/\aff{L}, \aff{L})$ corresponding to 
\begin{align}
\label{eq:normalBasis}
\begin{split}
&\begin{blockarray}{cccc}
\begin{block}{[cccc]}
  0 & \cdots & 0 & 1 \\
  0 & \cdots & 0 & 0 \\
\end{block}
\end{blockarray} \, ,   \quad
 \begin{blockarray}{ccccccccc} 
\begin{block}{[ccccccccc]}
  0 & \cdots & 0 & -C_1 & -c_{2,1} & -c_{3,1}  & \cdots & -c_{m-2,1} & 0 \\
  0 & \cdots & 0 & 0        & 1            & 0             & \cdots & 0 & 0 \\
\end{block}
\end{blockarray} \, , \, \\ 
&\ldots \, , 
\begin{blockarray}{cccccc}
\begin{block}{[cccccc]}
  0 & \cdots & 0 & -C_{m-3} & -c_{m-2,m-3} & 0  \\
  0 & \cdots & 0 & 0   & 1 & 0\\
\end{block}
\end{blockarray} \, , \quad
\begin{blockarray}{ccccc}
\begin{block}{[ccccc]}
  0 & \cdots & 0 & -C_{m-2} & 0  \\
  0 & \cdots & 0 & 0 &  1\\
\end{block}
\end{blockarray} \, .
\end{split}
\hspace*{-3mm}
\end{align}

Since the matrix entry in row $2$ and column $n-m+1$ is zero in every matrix in $N^\ast_{\mathcal{L}_m(X),L}$, the only rank one matrices in $N^\ast_{\mathcal{L}_m(X),L}$ are scalar multiples of the first matrix in~\eqref{eq:normalBasis}.
By our choice of coordinates, the linear map $\varphi: \KK^{n+1}/\aff{L} \to \aff{L}$ corresponding to the first matrix in~\eqref{eq:normalBasis} has image $\aff{p_L}$ and kernel $\aff{\TT_{X,p_L}}/\aff{L}$.
Thus, set-theoretically the intersection of $\PP (N^\ast_{\mathcal{L}_m(X),L})$ with the Segre variety $\Seg(\KK^{n+1}/\aff{L}, \aff{L})$ 
consists just of one point as claimed in Theorem~\ref{thm:higherContactLines}.

We can either see from~\eqref{eq:normalBasis} or directly from~\eqref{eq:HCC_m} that
$$
N^\ast_{\mathcal{L}_m(X),L} \subset \lbrace \psi \in \Hom(\KK^{n+1}/\aff{L}, \aff{L}) \mid \aff{\TT_L \mathcal{C}_{m,p_L}(X)}/\aff{L} \subset \ker \psi \rbrace. 
$$
Hence, we can embed $N^\ast_{\mathcal{L}_m(X),L}$ canonically into $\Hom(\KK^{n+1}/\aff{\TT_L \mathcal{C}_{m,p_L}(X)}, \aff{L})$ and denote the image of this embedding by $N'$.
In our coordinates, this embedding simply forgets the first $n-m$ columns of the matrices in~\eqref{eq:normalBasis}.
As before, $\PP(N')$ intersects the Segre variety $\Seg(\KK^{n+1}/\aff{\TT_L \mathcal{C}_{m,p_L}(X)}, \aff{L})$ set-theoretically at one point.
Since the codimension of this Segre variety in its ambient space $\PP(\Hom(\KK^{n+1}/\aff{\TT_L \mathcal{C}_{m,p_L}(X)}, \aff{L}))$ is $m-2$ and the dimension of $\PP(N')$ is also $m-2$,
the intersection multiplicity at the unique point of intersection is $\deg \Seg(\KK^{n+1}/\aff{\TT_L \mathcal{C}_{m,p_L}(X)}, \aff{L}) = m-1$.
Thus, also the intersection multiplicity of $\PP(N^\ast_{\mathcal{L}_m(X),L})$ and $\Seg(\KK^{n+1}/\aff{L}, \aff{L})$ at their unique point of intersection is $m-1$.
\end{proof}

\begin{proof}[Proof of Lemma~\ref{lem:higherContactPhi}.]
We consider the rational map $\mathcal{L}_m(X) \dashrightarrow X$, $L \mapsto p_L$.
It restricts to $\mathcal{L}_{m,p_L}(X) \dashrightarrow \lbrace p_L \rbrace$, which shows  $\Phi_L(T_{\mathcal{L}_{m,p_L}(X),L}) \subset T_{\lbrace p_L \rbrace, p_L} = \lbrace 0 \rbrace$.
Since $\mathcal{L}_{m,p_L}(X)$ and $\ker \Phi_L$ both have dimension $n-m$, we have derived the equality $T_{\mathcal{L}_{m,p_L}(X),L} = \ker \Phi_L$.
The image of every $\varphi \in T_{\mathcal{L}_{m,p_L}(X),L}$ is contained in $\aff{\TT_L \mathcal{C}_{m, p_L}(X)}/\aff{L}$ and its kernel contains $\aff{p_L}$.
Since both vector spaces in~\eqref{eq:HCker} have the same dimension, the equality in~\eqref{eq:HCker} is proven.

For the remaining assertions, we choose coordinates $x_0,\ldots, x_n$ on $\PP^n$ such that we can express the flag in~\eqref{eq:higherContactFlag} with the following zero loci:
\begin{align*}
p_L &= Z(x_1, \ldots, x_n), \\
L &= Z(x_2, \ldots, x_n), \\
\TT_L \mathcal{C}_{k, p}(X) &= Z(x_{n-k+2}, \ldots, x_n) \quad \text{ for } 2 \leq k \leq m.
\end{align*}
In the following, we work in the affine chart $\PP^n \setminus Z(x_0) \isom \KK^n$ with the standard basis $e_1, \ldots, e_n$.
We extend this basis to a basis for $\KK^{n+1}$ by adding $e_0 \in \aff{p_L} \setminus \lbrace 0 \rbrace$.
Note that these coordinates are compatible with the coordinates in the proof of Theorem~\ref{thm:higherContactLines}.
We write $f(x) = f_1(x) + f_2(x) + \ldots $ for the defining equation of $X$ in the chosen affine chart, where $f_i$ is homogeneous of degree $i$.
By our choice of coordinates, there are constants $c_{i,j}$ such that the gradients of the $f_i$ satisfy
\begin{align}
\label{eq:gradients}
\begin{split}
\bigtriangledown f_1 &= (0, \ldots, 0, c_{1,n}) \text{ for } c_{1,n} \neq 0, \\
\bigtriangledown f_2 (e_1) &= (0, \ldots, 0, c_{2,n-1}, c_{2,n}) \text{ for } c_{2,n-1} \neq 0, \\
&\vdots \\
\bigtriangledown f_{m-1} (e_1) &= (0, \ldots, 0, c_{m-1, n-m+2}, \ldots, c_{m-1,n}) \text{ for } c_{m-1, n-m+2} \neq 0.
\end{split}
\end{align}

For every tangent direction $e_1, \ldots, e_{n-1}$ to $X \setminus Z(x_0)$ at $p_L$, we compute its fiber under $\Phi_L$ as follows. 
For each $1 \leq i \leq n-1$, we choose a path of points \mbox{$\gamma_i (t) = e_i \cdot t  + O(t^2) \in X \setminus Z(x_0)$}.
Along the path $\gamma_i$, we compute all possible paths of lines $L_i(t) \in \mathcal{L}_{m,\gamma_i (t)}(X)$ such that $L_i(0) = L$.
For this, we consider the Taylor expansion of $f$ around $\gamma_i (t)$.
Hence, we perform a linear change of coordinates $\tilde{x} := x - \gamma_i (t)$ such that $f(x) = f(\tilde{x} + \gamma_i (t)) =: F^{(i)}(\tilde{x})$.
Now we want to write again $F^{(i)} (\tilde{x}) = F^{(i)}_0 + F^{(i)}_1 (\tilde{x}) + F^{(i)}_2(\tilde{x}) + \ldots$, where $F^{(i)}_j$ is homogeneous of degree $j$ in~$\tilde{x}$.
For $j =1, \ldots, m$, we have $f_j (\tilde{x} + \gamma_i(t)) = f_j(\tilde{x}) + t \langle e_i, \bigtriangledown f_j(\tilde{x}) \rangle + O( t^2)$, where $\langle \cdot, \cdot \rangle$ denotes the nondegenerate bilinear form $(x,y) \mapsto \sum_k x_k y_k$.
This implies
$$F_j^{(i)}(\tilde{x}) = f_j(\tilde{x}) + t \langle e_i, \bigtriangledown f_{j+1}(\tilde{x}) \rangle + O(t^2) \quad\quad \text{ for } 1 \leq j \leq m-1.$$
The solutions of these $m-1$ polynomials are the directions with contact order of at least $m$ at $\gamma_i(t)$.
Since we are only interested in paths of lines that start at $L$, we want to compute those solutions of $F^{(i)}_1(\tilde{x}) = 0, \ldots, F^{(i)}_{m-1}(\tilde{x}) = 0$ that are of the form $v(t) := e_1 + d \cdot t + O(t^2)$ for some $d \in \KK^n$.
We get that 
\begin{align}
\label{eq:toSolve}
F^{(i)}_j(v(t)) = t \left\langle d, \bigtriangledown f_j(e_1) \right\rangle + t \left\langle e_i, \bigtriangledown f_{j+1} (e_1) \right\rangle + O(t^2),
\end{align}
since $f_j(e_1) = 0$ for $1 \leq j \leq m-1$.
Each solution $v(t)$ of $F^{(i)}_1(\tilde{x}) = 0,$ $\ldots,$ $F^{(i)}_{m-1}(\tilde{x}) = 0$ defines a path of lines $L_{i, v(t)}(t) \in \mathcal{L}_{m,\gamma_i(t)}(X)$ starting at $L$, where each line $L_{i, v(t)}(t)$ is spanned by $\gamma_i(t)$ and $\gamma_i(t) + v(t)$.
To this path of lines corresponds the tangent vector
 $\varphi_{i, v(t)} \in T_{\mathcal{L}_m(X),L} \subset \Hom(\aff{L}, \KK^{n+1}/\aff{L})$ defined by
\begin{align}
\label{eq:homomorphism}
\begin{split}
e_0 &\longmapsto {\gamma_i'(0)} + \aff{L} =  {e_i} + \aff{L}, \\ {e_1} &\longmapsto {\gamma_i'(0)} + {v'(0)} + \aff{L} = {e_i} + {d} + \aff{L}.
\end{split}
\end{align}
Since $\varphi_{i, v(t)}$ depends only on $d$, we write $\varphi_{i, d} := \varphi_{i, v(t)}$.
The fiber of the tangent direction $e_i$ under $\Phi_L$ is spanned by all $\varphi_{i, d}$ such that $d \in \KK^n$ is a solution of 
\begin{align}
\label{eq:systemOfEquations}
F_1^{(i)}(v(t))=0, \ldots, F_{m-1}^{(i)}(v(t))=0.
\end{align}

\paragraph*{$\bs{i=1:}$} The fiber of the tangent direction $e_1$ under $\Phi_L$ is $\Phi_L^{-1}(T_{L,p_L})$. 
By~\eqref{eq:toSolve}, \eqref{eq:gradients} and Euler's relation for homogeneous polynomials, we have 
\begin{align*}
%\label{eq:F^1}
\begin{split}
F_1^{(1)} (v(t)) &= c_{1,n} d_n \cdot t + O(t^2), \\
F_2^{(1)} (v(t)) &= \left( c_{2,n-1} d_{n-1} + c_{2,n}d_n \right) \cdot t + O(t^2), \\
&\vdots \\
F_{m-2}^{(1)} (v(t)) &= \left( c_{m-2,n-m+3} d_{n-m+3} + \ldots +  c_{m-2,n}d_n \right) \cdot t + O(t^2), \\
F_{m-1}^{(1)} (v(t)) &= \left( c_{m-1,n-m+2} d_{n-m+2} + \ldots +  c_{m-1,n}d_n + m f_m(e_1) \right) \cdot t + O(t^2).
\end{split}
\end{align*}
Solving for $d$ implies $d_n = 0, d_{n-1}=0, \ldots, d_{n-m+3}=0$.
Since $L \in \mathcal{L}_m(X)$ was chosen generally, it has exactly intersection multiplicity $m$ at $p_L$.
Thus, we have that $f_m(e_1) \neq 0$ and $d_{n-m+2}$ must be a non-zero constant.
Furthermore, $d_1, \ldots, d_{n-m+1}$ are arbitrary. 
Since $\Phi_L^{-1}(T_{L,p_L})$ is spanned by all $\varphi_{1,d}$ for these solutions $d$, where $\varphi_{1,d}$ is defined by~\eqref{eq:homomorphism}, we have shown~\eqref{eq:HCL}.

\paragraph*{$\bs{2 \leq i \leq n-m+1:}$}
The fibers of the tangent directions $e_1, \ldots, e_{n-m+1}$ under $\Phi_L$ span $\Phi_L^{-1}(T_L \mathcal{C}_{m,p_L}(X))$.
For $1 \leq j \leq m-2$, we have $F^{(i)}_j(v(t)) = F^{(1)}_j(v(t))$.
This implies again $d^{(i)}_n = 0, d^{(i)}_{n-1}=0, \ldots, d^{(i)}_{n-m+3}=0$ for a solution $d^{(i)}$ of~\eqref{eq:systemOfEquations}.
Furthermore, the equality 
$$F_{m-1}^{(i)} (v(t)) = ( c_{m-1,n-m+2} d_{n-m+2} + \ldots +  c_{m-1,n}d_n + \tfrac{\partial f_m}{\partial x_i}(e_1) ) \cdot t + O(t^2)$$
shows that $d^{(i)}_{n-m+2}$ is some constant.
As before, $d^{(i)}_1, \ldots, d^{(i)}_{n-m+1}$ are arbitrary. 
Since $\Phi_L^{-1}(T_L \mathcal{C}_{m,p_L}(X))$ is spanned by all $\varphi_{i,d^{(i)}}$ for $1 \leq i \leq n-m+1$ and solutions $d^{(i)}$ of~\eqref{eq:systemOfEquations}, where $\varphi_{i,d^{(i)}}$ is defined by~\eqref{eq:homomorphism}, we have shown~\eqref{eq:HCC_m}.

\paragraph*{$\bs{i = n-k}$ for $\bs{1 \leq k \leq m-2:}$}
For each $1 \leq l \leq m-2$, the fibers of the  tangent directions $e_1$, $\ldots$, $e_{n-m+l+1}$ under $\Phi_L$ span $\Phi_L^{-1}(T_L \mathcal{C}_{m-l,p_L}(X))$.
For  $1 \leq j \leq k-1$, we have $F^{(i)}_j(v(t)) = F^{(1)}_j(v(t))$.
So a solution $d^{(i)}$ of~\eqref{eq:systemOfEquations} satisfies $d^{(i)}_n = 0, \ldots, d^{(i)}_{n-k+2}=0$.
Since $$F_{k}^{(i)} (v(t)) = ( c_{k,n-k+1} d_{n-k+1} + \ldots +  c_{k,n}d_n + c_{k+1, n-k} ) \cdot t + O(t^2)$$ and the two constants $c_{k,n-k+1}$ and $c_{k+1, n-k}$ are non-zero, we get that $d^{(i)}_{n-k+1}$ must be a non-zero constant. 
Finally, we have for $k+1 \leq j \leq m-1$ that $$F_j^{(i)}(v(t)) = ( c_{j,n-j+1} d_{n-j+1} + \ldots +  c_{j,n}d_n + \tfrac{\partial f_{j+1}}{\partial x_i}(e_1) ) \cdot t + O(t^2),$$ which implies that $d^{(i)}_{n-k}, \ldots, d^{(i)}_{n-m+2}$ are some constants.
As before, the remaining entries $d^{(i)}_1, \ldots, d^{(i)}_{n-m+1}$ are arbitrary. 
Since $\Phi_L^{-1}(T_L \mathcal{C}_{m-l,p_L}(X))$ is spanned by all $\varphi_{i,d^{(i)}}$ for $1 \leq i \leq n-m+l+1$ and solutions $d^{(i)}$ of~\eqref{eq:systemOfEquations}, where $\varphi_{i,d^{(i)}}$ is defined by~\eqref{eq:homomorphism}, we have shown~\eqref{eq:HCC_m-l1} -- \eqref{eq:HCC_m-l3}.
\end{proof}

\section{Isotropic Varieties}
\label{sec:Iso}

We investigate a dual notion to \emph{coisotropic varieties}.
Instead of imposing rank one conditions on the conormal spaces of a subvariety of a Grassmannian, we require such conditions to hold for its tangent spaces.

\begin{defn}
\label{defn:isotropy}
An irreducible subvariety $\Sigma \subset \Gr(\ell, \PP^n)$ of dimension $d \geq 1$ is \emph{isotropic} if, for every $L \in \Reg(\Sigma)$, the tangent space of $\Sigma$ at $L$ is spanned by rank one homomorphisms or is in the Zariski closure of the set of such spaces, i.e.,
$$\PP(T_{\Sigma,L}) \in \Sec_{d-1}(\Seg(\aff{L}, \, \KK^{n+1}/\aff{L})).$$
Moreover, $\Sigma$ is \emph{strongly isotropic} if, for every $L \in \Reg(\Sigma)$, the rank of every homomorphism in the tangent space of $\Sigma$ at $L$ is at most one, i.e.,
$$\PP(T_{\Sigma,L}) \subset \Seg(\aff{L}, \, \KK^{n+1}/\aff{L}).$$
\end{defn}

First, we prove that all subvarieties of Grassmannians of low enough codimension are isotropic.

\begin{proof}[Proof of Proposition~\ref{prop:keyPropIso}]
Let $\Sigma \subset \Gr(\ell, \PP^n)$ be a subvariety with $\codim \Sigma \leq n-1$.
The dimension $d$ of $\Sigma$ is at least $(\ell+1)(n-\ell)-(n-1) = \ell(n-\ell-1)+1$. 
Moreover, the codimension of $\Seg(\aff{L},\KK^{n+1}/\aff{L})$ in $\PP (\Hom(\aff{L}, \KK^{n+1}/\aff{L}))$ is $\ell(n-\ell-1)$ for every $L \in \Gr(\ell, \PP^n)$.
By Lemma~\ref{lem:grSecSpans}, we have 
$$\Sec_{d-1}(\Seg(\aff{L}, \KK^{n+1}/\aff{L})) =\Gr(d-1, \PP(\Hom(\aff{L},\KK^{n+1}/\aff{L}))).$$
In particular, $\Sigma$ is isotropic.
\end{proof}

\subsection{Isotropic Curves}
\label{ssec:isotropicCurves}

In this section, we show that the isotropic curves are exactly the curves of osculating spaces to projective curves.
For a detailed treatment of osculating spaces and osculating bundles of projective curves we refer to~\cite{pieneCurves}.
We only give a brief and non-exhaustive summary here.
For a nondegenerate irreducible curve $C \subset \PP^n$ and an integer $0 \leq k \leq n$, we define the \emph{osculating $k$-space} $L_k(x)$ at a point $x \in \Reg(C)$ to be the $k$-dimensional subspace of $\PP^n$ which has the highest order of contact with $C$ at $x$.
This order of contact of $L_k(x)$ with $C$ is at least $k+1$, and at a general point $x \in C$ it is exactly $k+1$.
Those osculating $k$-spaces with a higher contact order are called \emph{hyperosculating} or \emph{stationary} $k$-spaces.
For example, the tangent line at a flex point is a stationary tangent. 
For $k < n$, we define the \emph{curve $\Osc_k(C) \subset \Gr(k, \PP^n)$ of osculating $k$-spaces} of $C$ as the Zariski closure of the set of all osculating $k$-spaces $L_k(x)$ at smooth points $x \in C$. 

\begin{rem}
Clearly, $\Osc_0(C) = C$.
The curve $\Osc_1(C)$ consists of the tangent lines of $C$. 
It was denoted by $\mathcal{T}(C)$ for $C \subset \PP^3$ in Example~\ref{ex:tangentLinesToCurveRationalMap}.
The curve $\Osc_{n-1}(C) \subset \Gr(n-1, \PP^n) \isom (\PP^n)^\ast$ is also known as the \emph{dual curve} of $C$.
For any nondegenerate curve $C \subset \PP^n$, the dual curve of $\Osc_{n-1}(C)$ is again the original curve~ $C$~\cite[Thm.~5.1]{pieneCurves}. 
$\hfill\diamondsuit$
\end{rem}

Of course we can also define osculating spaces to degenerate curves, which we will use in the formulation of Theorem~\ref{thm:isotropicCurves}.
If a curve $C \subset \PP^n$ spans $\PP^m \subset \PP^n$ for $m < n$, we define the osculating $k$-spaces for $0 \leq k \leq m$ inside of $\PP^m$ such that we have $m$ curves $\Osc_k(C) \subset \Gr(k,\PP^n)$ for $0 \leq k < m$.

Now we consider an irreducible isotropic curve $\Sigma \subset \Gr(\ell, \PP^n)$.
Since each curve in $\PP^n$ or $(\PP^n)^\ast$ is trivially isotropic, we assume $1 \leq \ell \leq n-2$.
For each $L \in \Reg(\Sigma)$, the tangent space $T_{\Sigma, L}$ is spanned by a rank one homomorphism $\varphi_L: \aff{L} \to \KK^{n+1} / \aff{L}$.
Hence, to every $L \in \Reg(\Sigma)$ we associate unique linear spaces $L^- \in \Gr(\ell-1, \PP^n)$ and $L^+ \in \Gr(\ell+1, \PP^n)$ such that $L^- \subset L \subset L^+$,  $\ker \varphi_L = \aff{L^-}$, and $\im \varphi_L = \aff{L^+} / \aff{L}$.
This provides us with a rational map  $\tau: \Sigma \dashrightarrow \Gr(\ell-1, \PP^n) \times \Gr(\ell+1, \PP^n)$.
We denote the projections of the codomain onto the first or second factor by $\pi^-$ or $\pi^+$, respectively.
Moreover, for $\star \in \lbrace +, - \rbrace$, we let $\varpi^\star := \pi^\star \circ\tau$ and define $\Sigma^\star \subset \Gr(\ell \star 1, \PP^n)$ as the Zariski closure of the image of $\varpi^\star$.
We summarize these maps in the following commutative diagram:
\begin{center}
\begin{small}
\begin{tikzcd}[framed, row sep = small]
& & \Sigma^- 
\arrow[d, hook]\\
& & \Gr(\ell-1, \PP^n) \\
\Sigma
\arrow[r, dashed, "\tau"] 
\arrow[rrdd, dashed, bend right = 16, "\varpi^+" swap] 
\arrow[rruu, dashed, bend left = 16, "\varpi^-"] &
\Gr(\ell-1, \PP^n) \times \Gr(\ell+1, \PP^n)
\arrow[dr, two heads, "\pi^+" {swap, near end}] 
\arrow[ur, two heads, "\pi^-" near end] \\
& & \Gr(\ell+1, \PP^n) \\
& & \Sigma^+
\arrow[u, hook]
\end{tikzcd}
\end{small}
\end{center}
Since $\Sigma$ is an irreducible curve, $\Sigma^-$ and $\Sigma^+$ are both irreducible and have each dimension zero or one. 
If $\Sigma^-$ is a point, then all $L \in \Sigma$ contain the same  $(\ell-1)$-dimensional subspace.
Analogously, if $\Sigma^+$ is a point, then all $L \in \Sigma$ are contained in the same $(\ell+1)$-dimensional projective space.
If $\Sigma^-$ or $\Sigma^+$ is a curve, we show now that it is isotropic and that we can recover $\Sigma$ from it.

\begin{lem}
\label{lem:sigmaMinus}
If $\Sigma^-$ is a curve, then, for each $L \in \Reg(\Sigma)$ such that $L^- \in \Reg(\Sigma^-)$, the image of every homomorphism in $T_{\Sigma^-, L^-}$ is contained in $\aff{L} / \aff{L^-}$.
In particular, $\Sigma^-$ is isotropic.
\end{lem}

\begin{proof}
Let $L \in \Sigma$ be general such that the differential $D_L \varpi^-: T_{\Sigma, L} \to T_{\Sigma^-, L^-}$ of $\varpi^-$ at $L$ is bijective. 
For each $\varphi \in T_{\Sigma, L}$, we have $\varphi |_{\aff{L^-}} = (D_L\varpi^-(\varphi) \mod \aff{L})$ by Corollary~\ref{cor:keyLemmaTangentCorrespondence}.
Since the kernel of every homomorphism in $T_{\Sigma, L}$ contains~$\aff{L^-}$, we have $(\psi \mod \aff{L}) = 0$ for every $\psi \in T_{\Sigma^-, L^-}$.
Thus, the image of each homomorphism in $T_{\Sigma^-, L^-}$ is contained in $\aff{L} / \aff{L^-}$.
\end{proof}

\begin{lem}
\label{lem:sigmaPlus}
If $\Sigma^+$ is a curve, then, for each $L \in \Reg(\Sigma)$ such that $L^+ \in \Reg(\Sigma^+)$, the kernel of every homomorphism in $T_{\Sigma^+, L^+}$ contains $\aff{L}$.
In particular, $\Sigma^+$ is isotropic.
\end{lem}

\begin{proof}
The proof of this statement is completely dual to the proof of Lemma~\ref{lem:sigmaMinus}.
We consider a general $L \in \Sigma$ such that the differential $D_L\varpi^+$ is bijective. 
For each $\varphi \in T_{\Sigma, L}$, we have $D_L\varpi^+(\varphi) |_{\aff{L}} = (\varphi \mod \aff{L^+})$ by Corollary~\ref{cor:keyLemmaTangentCorrespondence}.
Since the image of every homomorphism in $T_{\Sigma, L}$ is contained in $\aff{L^+}/\aff{L}$, we have that $\aff{L} \subset \ker \psi$ for every $\psi \in T_{\Sigma^+, L^+}$.
\end{proof}

For every curve $\Sigma \subset \PP^n$, we define $\Sigma^+ \subset \Gr(1, \PP^n)$ as above. 
Note that $\Sigma^+$ is the curve of tangent lines to $\Sigma$.
Dually, for a curve $\Sigma \subset \Gr(n-1, \PP^n) \isom (\PP^n)^\ast$, we define $\Sigma^- \subset \Gr(n-2, \PP^n)$.
This allows us to formulate the following:

\begin{cor}
Let $\Sigma \subset \Gr(\ell, \PP^n)$ be an irreducible isotropic curve.
If $\ell \geq 1$  and $\Sigma^-$ is a curve, then $(\Sigma^-)^+ = \Sigma.$
Dually, if $\ell \leq n-2$ and $\Sigma^+$ is a curve, then $(\Sigma^+)^- = \Sigma.$
\end{cor}

\begin{proof}
The first part follows immediately from Lemma~\ref{lem:sigmaMinus}, which shows that $(L^-)^+ = L$ for a general $L \in \Sigma$.
Similarly, Lemma~\ref{lem:sigmaPlus} implies the second part.
\end{proof}

So far we have shown that we can associate a maximal sequence of isotropic curves in different Grassmannians to each isotropic curve $\Sigma \subset \Gr(\ell, \PP^n)$, where $0 \leq \ell \leq n-1$.
By this, we mean a sequence of isotropic curves
\begin{align}
\label{eq:sequenceIsotropicCurves}
\begin{array}{ccccccccc}
\big( & \Sigma_{\ell_1} , & \Sigma_{\ell_1+1} , & \ldots, &  \Sigma_\ell , & \ldots , & \Sigma_{\ell_2-1} , & \Sigma_{\ell_2} & \big),\\
&\cap & \cap &&\cap && \cap &\cap &\\
&\scriptstyle \Gr(\ell_1, \PP^n) & \scriptstyle \Gr(\ell_1+1, \PP^n) &&  \scriptstyle \Gr(\ell, \PP^n)  && \scriptstyle \Gr(\ell_2-1, \PP^n) & \scriptstyle \Gr(\ell_2, \PP^n) &
\end{array}
\end{align}
where $\Sigma_\ell = \Sigma$, and $\ell_1 \in \lbrace 0, \ldots, \ell \rbrace$ is minimal and $\ell_2 \in \lbrace \ell, \ldots, n-1 \rbrace$ is maximal such that 
$\Sigma_i^+ = \Sigma_{i+1}$ holds for all $\ell_1 \leq i < \ell_2$ and $\Sigma_j^- = \Sigma_{j-1}$ holds for all $\ell_1 < j \leq \ell_2$.

If $\ell_2 < n-1$, then $\Sigma_{\ell_2}^+$ is a point in $\Gr(\ell_2+1, \PP^n)$, i.e., an $(\ell_2+1)$-dimensional subspace $P_2 \subset \PP^n$.
For each $\ell_1 \leq i \leq \ell_2$, we have that every $L \in \Sigma_i$ is contained in $P_2$.
Hence, the whole sequence~\eqref{eq:sequenceIsotropicCurves} of isotropic curves can be embedded into Grassmannians of subspaces of $P_2$ via $\Sigma_i \hookrightarrow \Gr(i, P_2), L \mapsto L$.

Dually, if $\ell_1 > 0$, then $\Sigma_{\ell_1}^-$ is a point in $\Gr(\ell_1-1, \PP^n)$, i.e., an $(\ell_1-1)$-dimensional subspace $P_1 \subset \PP^n$.
For each $\ell_1 \leq i \leq \ell_2$, we have that every $L \in \Sigma_i$ contains $P_1$.
Thus, denoting by $\pi_{P_1}: \PP^n \dashrightarrow \PP(\KK^{n+1}/\aff{P_1}) \isom \PP^{n- \ell_1}$ the projection away from $P_1$, we can embed the whole sequence~\eqref{eq:sequenceIsotropicCurves} into smaller Grassmannians via 
$\Sigma_i \hookrightarrow \Gr(i-\ell_1, \PP^{n- \ell_1}), L \mapsto \pi_{P_1}(L)$.
We denote by $\pi_{P_1} (\Sigma_i)$ the image of $\Sigma_i$ under this embedding.

Therefore, we can always assume that $\ell_1 = 0$ and $\ell_2 = n-1$.
In this case, $\Sigma_k = \Osc_k(\Sigma_0)$ for every $0 \leq k \leq n-1$, as the following theorem shows.
Recall for this theorem that we view the empty set as a projective space with dimension~$-1$.

\begin{thm}
\label{thm:isotropicCurves}
\begin{enumerate}
\item For each irreducible curve $\Sigma_0 \subset \PP^n$ and each $0 \leq k \leq \ell_2$, we have that $\Osc_k(\Sigma_0) = \Sigma_k$.
In particular, the curve $\Osc_k(\Sigma_0)$ of osculating $k$-planes of $\Sigma_0$ is isotropic.
\item For each irreducible isotropic curve $\Sigma \subset \Gr(\ell, \PP^n)$, there is a unique subspace $P_1 \subset \PP^n$ and a unique irreducible curve $C \subset \PP^{n-\dim P_1 - 1}$ such that
$-1 \leq \dim P_1 < \ell$, every $L \in \Sigma$ contains $P_1$, and $\pi_{P_1}(\Sigma) = \Osc_{\ell-\dim P_1-1}(C)$.
\end{enumerate}
\end{thm}

Note that this theorem is a reformulation of Theorem~\ref{thm:isoCurvesFULL}.

\begin{proof}[Proof of Theorem~\ref{thm:isoCurvesFULL} and Theorem~\ref{thm:isotropicCurves}]
Let us first consider the second part of Theorem~\ref{thm:isotropicCurves}.
We write $\Sigma_\ell = \Sigma$ and associate the maximal sequence~\eqref{eq:sequenceIsotropicCurves} of isotropic curves.
If $\ell_1 = 0$, we set $P_1 := \emptyset$.
Otherwise, $\ell_1 >0 $ and $P_1 := \Sigma_{\ell_1}^-$.
Note that $P_1$ is the maximal subspace of $\PP^n$ with the property that it is contained in every $L \in \Sigma$.
In any case, we define $C := \pi_{P_1}(\Sigma_{\ell_1}) \subset \PP^{n-\ell_1}$.
Now, $\pi_{P_1}(\Sigma) = \Osc_{\ell-\ell_1}(C)$ follows from the first part of Theorem~\ref{thm:isotropicCurves}.

Hence, it is enough to show the first assertion of this theorem.
For this, we consider an irreducible curve $\Sigma_0 \subset \PP^n$ and a point $p \in \Reg(\Sigma_0)$.
We choose coordinates such that $p = (1:0:\ldots:0)$.
Without loss of generality, we may work in the affine chart $\PP^n \setminus Z(x_0)$.
Since $\Sigma_0$ is smooth around $p$, there is a local analytic isomorphism $f$ from a neighborhood of the origin in $\KK^1$ to a neighborhood of the point $p$ in~$\Sigma_0$.
The map $f$ has the form $f(t) = \left( f_1(t), \ldots, f_n(t) \right)$ for some $f_1, \ldots, f_n \in \CC[\![ t ]\!]$ with $f(0) = (0, \ldots, 0 )$.
In our affine chart, the osculating $k$-plane at $f(\varepsilon)$ is the affine span of $f(\varepsilon)$, $\tfrac{\partial f}{\partial t}(\varepsilon)$, \ldots, $\tfrac{\partial^k f}{\partial^k t}(\varepsilon)$. 
We define 
\begin{align*}
A_k(\varepsilon) := \left[ \begin{array}{cccc}
1 & f_1(\varepsilon) & \cdots & f_n(\varepsilon) \\
1 & \frac{\partial f_1}{\partial t}(\varepsilon) & \cdots & \frac{\partial f_n}{\partial t}(\varepsilon) \\
\vdots & \vdots && \vdots \\
1 & \frac{\partial^k f_1}{\partial^k t}(\varepsilon) & \cdots & \frac{\partial^k f_n}{\partial^k t}(\varepsilon) 
\end{array} \right] \in \CC^{(k+1) \times (n+1)}
\end{align*}
such that the rowspace $L_k(\varepsilon) \in \Gr(k, \PP^n)$ of $A_k(\varepsilon)$ is an osculating $k$-plane of~$\Sigma_0$.
In particular, $L_k(0)$ is the osculating $k$-plane at $p$.
The tangent line at $A_k(\varepsilon)$ to the curve in $\KK^{(k+1) \times (n+1)}$ parametrized by $t \mapsto A_k(t)$ is the affine span of $A_k(\varepsilon)$ and
$$\left[\begin{array}{cccc} 0 & \frac{\partial f_1}{\partial t}(\varepsilon) & \cdots & \frac{\partial f_n}{\partial t}(\varepsilon) \\
\vdots & \vdots && \vdots \\ 
0 & \frac{\partial^{k+1} f_1}{\partial^{k+1} t}(\varepsilon) & \cdots & \frac{\partial^{k+1} f_n}{\partial^{k+1} t}(\varepsilon)  \end{array} \right].$$
Hence, by Corollary~\ref{cor:differentialStiefel}, the tangent line at $L_k(0)$ to the curve $\Osc_k(\Sigma_0)$, which is locally parametrized by $t \mapsto L_k(t)$, is spanned by 
\begin{align*}
\varphi: \aff{L_k(0)} &\longrightarrow \KK^{n+1} / \aff{L_k(0)}, \\
\left(1,f(0) \right) &\longmapsto 0 + \aff{L_k(0)}, \\
\left(1,\tfrac{\partial f}{\partial t}(0) \right) & \longmapsto 0 + \aff{L_k(0)}, \\
\vdots \\
\left(1, \tfrac{\partial^{k-1}f}{\partial^{k-1}t}(0) \right) &\longmapsto 0 + \aff{L_k(0)}, \\
\left(1, \tfrac{\partial^{k}f}{\partial^{k}t}(0) \right) &\longmapsto \left(0,\tfrac{\partial^{k+1}f}{\partial^{k+1}t}(0) \right) + \aff{L_k(0)}. 
\end{align*}
Since $p$ was an arbitrary smooth point of $\Sigma_0$, this shows that the curve $\Osc_k(\Sigma_0)$ is isotropic with $(\Osc_k(\Sigma_0))^+=\Osc_{k+1}(\Sigma_0)$ and $(\Osc_k(\Sigma_0))^-=\Osc_{k-1}(\Sigma_0)$.
Now the assertion $\Osc_k(\Sigma_0) = \Sigma_k$ follows by induction on $k$. It is trivial for $k=0$.
For $k >0$, the induction hypothesis yields $(\Osc_k(\Sigma_0))^- = \Osc_{k-1}(\Sigma_0) = \Sigma_{k-1}$, which then implies $\Osc_k(\Sigma_0) = ((\Osc_k(\Sigma_0))^-)^+ = \Sigma_{k-1}^+ = \Sigma_k$.
\end{proof}

\begin{example}
\label{ex:developableSurfacesAreIsotropicCurves}
Let $\Sigma \subset \Gr(1, \PP^3)$ be an irreducible isotropic curve and let $\mathcal{S}_\Sigma$ in  $\PP^3$ be the union of all lines on $\Sigma$.
If $\Sigma^- \subset \PP^3$ is a nondegenerate curve, the surface $\mathcal{S}_\Sigma$ is the tangent developable of $\Sigma^-$.
The projectively dual variety $\mathcal{S}_\Sigma^\vee$ is the dual curve $\Osc_2(\Sigma^-) = \Sigma^+$ of $\Sigma^-$ (see~\cite[page 111]{pieneCuspProj}).
The ruled surface $\mathcal{S}_\Sigma$ is cuspidal along $\Sigma^-$, and $\Sigma^-$ is classically known as its \emph{edge of regression}. We summarize these duality relations in Figure~\ref{fig:dualitySpaceCurves}, where $E(S) \subset \PP^3$ denotes the edge of regression of a developable surface $S \subset \PP^3 $. 

The situation degenerates when $\Sigma^- \subset \PP^3$ is a plane curve. In this case $\mathcal{S}_\Sigma$ is the plane containing $\Sigma^-$, and $(\Sigma^-)^\vee$ is a cone with vertex $\mathcal{S}_\Sigma^\vee = \Sigma^+$.
Dually, if $\Sigma^-$ is a point and $\Sigma^+$ is a plane curve, then $\mathcal{S}_\Sigma$ is a cone with vertex $\Sigma^-$. Finally, if both $\Sigma^-$ and $\Sigma^+$ are points, then $\mathcal{S}_\Sigma$ is the plane $(\Sigma^+)^\vee$.

In any case, the ruled surface $\mathcal{S}_\Sigma$ can be either a plane or the dual of a curve. 
So shortly put, the curve $\Sigma \subset \Gr(1, \PP^3)$ is isotropic if and only if the surface $\mathcal{S}_\Sigma$ ruled by the lines on $\Sigma$ is developable.
$\hfill\diamondsuit$
\end{example}

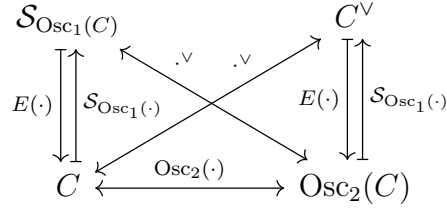
\begin{figure}
\centering
\begin{tikzcd}[column sep = huge, row sep = huge]
\mathcal{S}_{\Osc_1(C)}
\arrow[dr, leftrightarrow, "\cdot^\vee" near start]
\arrow[d, maps to, shift right=1mm, "E(\cdot)" left] & 
C^\vee 
\arrow[dl, leftrightarrow, "\cdot^\vee\hspace*{-1mm}" {swap, pos = 0.30}]
\arrow[d, maps to, shift right=1mm, "E(\cdot)" left] \\
C
\arrow[u, maps to, shift right=1mm, "\mathcal{S}_{\Osc_1(\cdot)}" right]
\arrow[r, leftrightarrow, "\Osc_2(\cdot)"] &
\Osc_2(C)
\arrow[u, maps to, shift right=1mm, "\mathcal{S}_{\Osc_1(\cdot)}" right]
\end{tikzcd}
\caption{Duality relations of a nondegenerate curve $C \subset \PP^3$.}
\label{fig:dualitySpaceCurves}
\end{figure}

\subsection{Strongly Isotropic Varieties}
\label{ssec:stronglyIsotropic}

In this subsection, we show that each strongly isotropic variety is either a curve or a subvariety of an $\alpha$- or $\beta$-\emph{variety}.

\begin{defn}
For $1 \leq \ell \leq n$ and $P_1 \in \Gr(\ell-1, \PP^n)$, we call 
$$\alpha(P_1) := \lbrace L \in \Gr(\ell, \PP^n) \mid P_1 \subset L \rbrace$$
the \emph{$\alpha$-variety} of $P_1$.
Analogously, for $0 \leq \ell \leq n-2$ and $P_2 \in \Gr(\ell+1, \PP^n)$, the \emph{$\beta$-variety} of $P_2$ is 
$$\beta(P_2) := \lbrace L \in \Gr(\ell,\PP^n) \mid L \subset P_2 \rbrace.$$
\end{defn}

Note that $\alpha(P_1)$ is isomorphic to $\Gr(0, \PP^{n-\ell}) = \PP^{n-\ell}$, and that
$\beta(P_2)$ is isomorphic to $\Gr(\ell, \PP^{\ell+1}) \isom (\PP^{\ell+1})^\ast$.

\begin{rem}
\label{rem:tangentSpacesAlphaBeta}
The tangent space of the $\alpha$-variety $\alpha(P_1)$ at a point $L$ is the $\alpha$-space $E_\alpha(\aff{P_1}) \subset \Hom(\aff{L}, \KK^{n+1}/\aff{L})$ (see Definition~\ref{def:alphaBetaSpace}).
Dually, the tangent space of the $\beta$-variety $\beta(P_2)$ at a point $L$ is the $\beta$-space $E_\beta (\aff{P_2} / \aff{L}) \subset \Hom(\aff{L}, \KK^{n+1}/\aff{L})$.
Conversely, given any $L \in \Gr(\ell, \PP^n)$ and an $\alpha$-space in $T_{\Gr(\ell, \PP^n),L}$, there is a unique $\alpha$-variety containing $L$ whose tangent space at $L$ is the given $\alpha$-space. The analogous assertion holds for $\beta$-spaces and $\beta$-varieties. 
$\hfill\diamondsuit$
\end{rem}

These observations show that subvarieties of $\alpha$- and $\beta$-varieties are strongly isotropic.
In fact, we can even show that all strongly isotropic varieties either are such subvarieties or have dimension one.
Note that we have seen in Subsection~\ref{ssec:isotropicCurves} that (strongly) isotropic curves are not necessarily contained in $\alpha$- or $\beta$-varieties.
The curve $\Osc_1(C) \subset \Gr(1, \PP^3)$ of tangent lines to a nondegenerate irreducible curve $C \subset \PP^3$ is such an example.

\begin{thm}
\label{thm:stronglyIsotropic}
\begin{enumerate}
\item Every subvariety of an $\alpha$- or $\beta$-variety is strongly isotropic.
\item Every irreducible strongly isotropic variety of dimension at least two is either a subvariety of a unique $\alpha$-variety or a subvariety of a unique $\beta$-variety.
\end{enumerate}
\end{thm}

This theorem is a more specific version of Theorem~\ref{thm:stronglyIsoFULL}.
As in the classification of strongly coisotropic varieties in Subsection~\ref{ssec:stronglyCoisotropic}, we use Lemma~\ref{lem:alphaBetaSegre} two distinguish two types of strongly isotropic varieties.

\begin{defCor}
Let $\Sigma$ be an irreducible strongly isotropic variety of dimension at least two.
Either each tangent space at a smooth point of $\Sigma$ is contained in a unique $\alpha$-space, or each tangent space at a smooth point of $\Sigma$ is contained in a unique $\beta$-space.
In the first case, we call $\Sigma$ \emph{strongly isotropic of $\alpha$-type}. 
In the latter case, we say that $\Sigma$ is \emph{strongly isotropic of $\beta$-type}. 
\end{defCor}

\begin{lem}
\label{lem:dimStronglyIso}
Every strongly isotropic variety in $\Gr(\ell, \PP^n)$ of $\alpha$-type has dimension at most $n-\ell$, 
and every strongly isotropic variety of $\beta$-type has dimension at most $\ell+1$.
\end{lem}

\begin{proof}
This follows from the fact that $\alpha$-spaces in $\Hom(\aff{L},\KK^{n+1}/\aff{L})$ have dimension $n-\ell$ and that the dimension of $\beta$-spaces is $\ell+1$, where $L \in \Gr(\ell, \PP^n)$.
\end{proof}

\begin{lem}
\label{lem:alphaBetaDualityIso}
\begin{enumerate}
\item A subvariety $\Sigma \subset \Gr(\ell,\PP^n)$ is strongly isotropic of $\alpha$-type if and only if $\Sigma^\perp \subset \Gr(n-\ell-1, (\PP^n)^\ast)$ is strongly isotropic of $\beta$-type.
\item Moreover, $\Sigma$ is contained in the $\alpha$-variety $\alpha(P_1)$ if and only if $\Sigma^\perp$ is contained in the $\beta$-variety $\beta(P_1^\vee)$.
\end{enumerate}
\end{lem}

\begin{proof}
First, we notice that $L$ is a smooth point of $\Sigma$ if and only if $L^\vee$ is a smooth point of $\Sigma^\perp$.
Secondly, a linear subspace $U \subset \aff{L}$ is contained in the kernel of $\varphi \in T_{\Sigma,L}$ if and only if the image of $\varphi^\ast \in T_{\Sigma^\perp,L^\vee}$ is contained in $(\aff{L}/U)^\ast$. This shows the first part of Lemma~\ref{lem:alphaBetaDualityIso}.
The second part is immediate.
\end{proof}

Thus, to prove Theorem~\ref{thm:stronglyIsotropic}, we only have to consider strongly isotropic varieties of $\alpha$-type.
To every strongly isotropic variety $\Sigma \subset \Gr(\ell,\PP^n)$ of $\alpha$-type, we can associate a variety $\Sigma_{\ker} \subset \Gr(\ell-1, \PP^n)$ as follows.
For every smooth point $L$ of $\Sigma$, there is a unique hyperplane $P_L \subset L$ such that every tangent vector $\varphi \in T_{\Sigma,L}$ satisfies $\aff{P_L} \subset \ker \varphi$.
Hence, we get a rational map
\begin{align}
\label{eq:PsiKer}
\begin{split}
\Psi_{\ker}: \Sigma &\longdashrightarrow \Gr(\ell-1, \PP^n), \\
\Reg(\Sigma) \ni L &\longmapsto P_L,
\end{split}
\end{align}
and we define $\Sigma_{\ker}$ to be the Zariski closure of its image.
Since $\Sigma$ is irreducible, so is $\Sigma_{\ker}$.
We will prove Theorem~\ref{thm:stronglyIsotropic} by showing that $\Sigma_{\ker}$ is a point.

\begin{lem}
\label{lem:SigmaKer}
Let $\Sigma \subset \Gr(\ell, \PP^n)$ be a strongly isotropic variety of $\alpha$-type.
\begin{enumerate}
\item For a general point $L$ of $\Sigma$, the image of every tangent vector $\phi \in T_{\Sigma_{\ker}, P_L}$ is contained in $\aff{L}/\aff{P_L}$.
%More specifically, for every smooth point $L=\PP(U)$ of $\Sigma$ such that $\PP(\mathcal{U}_L) \in \Reg(\Sigma_{\ker})$ we have that the image of every tangent vector $\phi \in \Hom(\mathcal{U}_L, V/\mathcal{U}_L)$ of $\Sigma_{\ker}$ at $\PP(\mathcal{U}_L)$ is contained in $U/\mathcal{U}_L$.
\item If the general fiber of $\Psi_{\ker}$ contains exactly one point, then $\Sigma_{\ker}$ is strongly isotropic of $\beta$-type.
Otherwise, $\Sigma_{\ker}$ must be a point.
%Besides, the dimension of $\Sigma_{\ker}$ is at most $\dim(\Sigma)$, and
%if $\dim(\Sigma_{\ker}) < \dim(\Sigma)$, then $\Sigma_{\ker}$ must be a point.
\end{enumerate}
\end{lem}

\begin{proof}
For a general point $L \in \Sigma$, the differential 
\begin{align*}
D_L \Psi_{\ker} : T_{\Sigma,L} \longrightarrow T_{\Sigma_{\ker},P_L}
\end{align*}
of $\Psi_{\ker}$ at $L$ is a surjection.
By Corollary~\ref{cor:keyLemmaTangentCorrespondence}, it sends a map $\varphi: \aff{L} \to \KK^{n+1}/\aff{L}$ in $T_{\Sigma,L}$ to a linear map $\phi: \aff{P_L} \to \KK^{n+1}/\aff{P_L}$
such that $\varphi |_{\aff{P_L}} = (\phi \mod \aff{L})$.
But since $\varphi |_{\aff{P_L}}$ is the zero-map, the image of $\phi$ must be contained in $\aff{L}/\aff{P_L}$.
In particular, the rank of $\phi$ is at most one.
Since $D_L \Psi_{\ker}$ is surjective, every $\phi \in T_{\Sigma_{\ker},P_L}$ has at most rank one and satisfies $\im \phi \subset \aff{L}/\aff{P_L}$.
This shows that $\Sigma_{\ker}$ is strongly isotropic of $\beta$-type if $\dim \Sigma_{\ker} \geq 2$.
Now we have shown the first assertion of Lemma~\ref{lem:SigmaKer} as well as the first part of the second assertion.

Finally, we assume that the general fiber of $\Psi_{\ker}$ contains more than one point.
We consider a general point $P \in \Sigma_{\ker}$ and two general points $L_1, L_2 \in \Psi_{\ker}^{-1} (P)$ in its fiber.
Both differentials $D_{L_1} \Psi_{\ker}$ and $D_{L_2} \Psi_{\ker}$ are surjective, and the image of every $\phi \in T_{\Sigma_{\ker},P}$ is contained in $\aff{L_1} / \aff{P} \cap \aff{L_2} / \aff{P} = \lbrace 0 \rbrace$.
Thus, the dimension of $\Sigma_{\ker}$ must be zero.
\end{proof}

\begin{proof}[Proof of Theorem~\ref{thm:stronglyIsoFULL} and Theorem \ref{thm:stronglyIsotropic}]
The first part follows immediately from Remark~\ref{rem:tangentSpacesAlphaBeta}.
For the second part, let $\Sigma \subset \Gr(\ell, \PP^n)$ be an irreducible strongly isotropic variety of dimension at least two.
Due to Lemma \ref{lem:alphaBetaDualityIso}, we can assume that $\Sigma$ is of $\alpha$-type.
Thus, we have a dominant rational map $\Psi_{\ker}: \Sigma \dashrightarrow \Sigma_{\ker}$ as in~\eqref{eq:PsiKer}.
Moreover, we assume for contradiction that $\Sigma_{\ker} \subset \Gr(\ell-1, \PP^n)$ is not a point.
By Lemma~\ref{lem:SigmaKer}, the general fiber of $\Psi_{\ker}$ consists of exactly one point and $\Sigma_{\ker}$ is strongly isotropic of $\beta$-type.
In particular, $\dim \Sigma_{\ker}  = \dim \Sigma$.

We denote by $X_{\ker} \subset \PP^n$ the variety swept out by all $P \in \Sigma_{\ker}$.
It is the image of the incidence correspondence $F_{\ker} := \lbrace (x,P) \in \PP^n \times \Sigma_{\ker} \mid x \in P \rbrace$ under the projection $\pi$ onto the first factor; 
so $\dim X_{\ker} \leq \dim F_{\ker} = \dim \Sigma + \ell-1 \leq n-1$ by Lemma~\ref{lem:dimStronglyIso}.
At a smooth point $(x,P) \in F_{\ker}$ such that $P \in \Reg(\Sigma_{\ker})$ the incidence correspondence has the tangent space
\begin{align}
\label{eq:tangentSpaceIncidenceSigma}
T_{F_{\ker}, (x,P)} = \left\lbrace (\phi, \varphi) \in \Hom(\aff{x}, \KK^{n+1}/\aff{x}) \times T_{\Sigma_{\ker},P} \mid \varphi |_{\aff{x}} = (\phi \mod \aff{P}) \right\rbrace,
\end{align}
since the inclusion ``$\subset$'' follows from Lemma~\ref{lem:tangentSpaceIncidenceGrassmannian} and both linear spaces in~\eqref{eq:tangentSpaceIncidenceSigma} have the same dimension.

%We distinguish two cases:
%\begin{enumerate}
%\item
Let $(x,P) \in F_{\ker}$ be general.
As the differential $D_{(x,P)}\pi: T_{F_{\ker},(x,P)} \to T_{X_{\ker},x}$ of $\pi$ at $(x,P)$ is surjective, we see from~\eqref{eq:tangentSpaceIncidenceSigma} that
$$T_{X_{\ker},x} = \left\lbrace \phi \in \Hom( \aff{x} , \KK^{n+1}/\aff{x}) \mid \exists \varphi \in T_{\Sigma_{\ker},P} : \varphi |_{\aff{x}} = (\phi \mod \aff{P})  \right\rbrace.$$
We denote by $L \in \Sigma$ the unique point in the fiber $\Psi_{\ker}^{-1}(P)$.
By Lemma~\ref{lem:SigmaKer}, the space $T_{X_{\ker},x}$ is a subset of $\mathrm{H}_{x,L} := \lbrace \phi \in \Hom( \aff{x}, \KK^{n+1}/\aff{x}) \mid \im \phi \subset \aff{L} / \aff{x} \rbrace$.
The dimension of $\mathrm{H}_{x,L}$ is $\ell$, which implies that the dimension of $X_{\ker}$ is at most $\ell$.
Since we assumed that $\Sigma_{\ker} \subset \Gr(\ell-1, \PP^n)$ is not a point, the dimension of $X_{\ker}$ must be exactly $\ell$, so $T_{X_{\ker},x} = \mathrm{H}_{x,L}$.
Moreover, the dimension of the fiber $\pi^{-1}(x)$ is $\dim F_{\ker}-\dim X_{\ker} = \dim \Sigma - 1 \geq 1$.
We pick another general point $(x,P')$ in this fiber and repeat the above argument.
We see that the unique point $L' \in \Sigma$ in the fiber $\Psi_{\ker}^{-1}(P')$ also satisfies $T_{X_{\ker},x} = \mathrm{H}_{x,L'}$.
In particular, we have $\mathrm{H}_{x,L}=\mathrm{H}_{x,L'}$, so $L= L'$ and $P=P'$, but this is a contradiction since $(x,P)$ and $(x,P')$ were generally chosen on the variety $\pi^{-1}(x)$ of dimension at least one.

%\item $\dim X_{\ker}  < \ell-1 + \dim \Sigma$:
%By Lemma \ref{lem:union}, the degree of $X_{\ker}$ is one.
%Hence, every $P \in \Sigma_{\ker}$ is contained in in the projective subspace $X_{\ker}$.
%In particular, the image of every $\varphi \in T_{\Sigma_{\ker},P}$ for a smooth point $P$ is contained in $\aff{X_{\ker}}/\aff{P}$.
%Thus, by Lemma \ref{lem:SigmaKer}, we either have that $\Sigma \subset \Gr(\ell, X_{\ker})$ or that the image of every $\varphi \in T_{\Sigma_{\ker},P}$ for smooth points $P \in \Sigma_{\ker}$ is $\lbrace 0 \rbrace$.
%Since we assumed that $\Sigma_{\ker}$ is not a point, we have $\Sigma \subset \Gr(\ell, X_{\ker})$.
%By Corollary~\ref{cor:dimStronglyIso}, the dimension of $\Sigma$ is at most $\dim(X_{\ker})-\ell$, which is strictly less than $\dim(\Sigma)-1$.
%This is clearly a contradiction.
%\end{enumerate}
Thus, we have shown that $\Sigma_{\ker}$ must be a point, i.e., an $(\ell-1)$-dimensional subspace $P \subset \PP^n$.
Every $L \in \Sigma$ contains $P$, so $\Sigma \subset \alpha(P)$.
\end{proof}

\paragraph{Acknowledgements}
We thank Peter Bürgisser, Emre Sertöz and Bernd Sturmfels for many helpful discussions. 
We are grateful to TU Berlin and MPI MIS in Leipzig for the hospitality in September/October 2017, where most ideas for this project were developed. 
We also thank the referee for useful comments concerning the exposition. 

Kathl\'en Kohn acknowledges partial support from
the Einstein Foundation Berlin.
\bibliographystyle{alpha}
\bibliography{literatur}

\begin{thebibliography}{CvdW37}

\bibitem[B\"17]{condition}
Peter B\"urgisser.
\newblock {Condition of Intersecting a Projective Variety with a Varying Linear
  Subspace}.
\newblock {\em SIAM Journal on Applied Algebra and Geometry}, 1(1):111--125,
  2017.

\bibitem[BKLS16]{our_computation}
Peter B\"urgisser, Kathl\'en Kohn, Pierre Lairez, and Bernd Sturmfels.
\newblock {Computing the Chow Variety of Quadratic Space Curves}.
\newblock In I.~Kotsireas, S.~Rump, and C.~Yap, editors, {\em Mathematical
  Aspects of Computer and Information Sciences, MACIS 2015, Berlin}, pages
  130--136, 2016.

\bibitem[BL13]{landsberg}
Jaros{\l}aw Buczy{\'n}ski and Joseph~M. Landsberg.
\newblock {Ranks of Tensors and a Generalization of Secant Varieties}.
\newblock {\em Linear Algebra and its Applications}, 438(2):668--689, 2013.

\bibitem[BL20]{schubert}
Peter B\"urgisser and Antonio Lerario.
\newblock {Probabilistic Schubert Calculus}.
\newblock {\em Journal für die reine und angewandte Mathematik}, 760:1--58,
  2020.

\bibitem[Cat14]{catanese}
Fabrizio~M.E. Catanese.
\newblock {Cayley Forms and Self-Dual Varieties}.
\newblock In {\em Proceedings of the Edinburgh Mathematical Society (Series
  2)}, volume~57, pages 89--109. Cambridge University Press, 2014.

\bibitem[Cay60]{cayley2}
Arthur Cayley.
\newblock {On a New Analytical Representation of Curves in Space}.
\newblock {\em The Quarterly Journal of Pure and Applied Mathematics},
  3:225--236, 1860.

\bibitem[CC01]{grSec}
Luca Chiantini and Marc Coppens.
\newblock {Grassmannians of Secant Varieties}.
\newblock In {\em Forum Mathematicum}, volume~13, pages 615--628. de Gruyter,
  2001.

\bibitem[CC03]{waring1}
Enrico Carlini and Jaydeep~V. Chipalkatti.
\newblock {On Waring’s Problem for Several Algebraic Forms}.
\newblock {\em Commentarii Mathematici Helvetici}, 78(3):494--517, 2003.

\bibitem[CC08]{grSecDim}
Ciro Ciliberto and Filip Cools.
\newblock {On Grassmann Secant Extremal Varieties}.
\newblock {\em Advances in Geometry}, 8(3):377--386, 2008.

\bibitem[CG03]{waring2}
Jaydeep~V. Chipalkatti and Anthony~V. Geramita.
\newblock {On Parameter Spaces for Artin Level Algebras}.
\newblock {\em The Michigan Mathematical Journal}, 51:187--208, 2003.

\bibitem[Chi04]{waring3}
Jaydeep~V. Chipalkatti.
\newblock {The Waring Locus of Binary Forms}.
\newblock {\em Communications in Algebra}, 32(4):1425--1444, 2004.

\bibitem[Coo09]{grSecSing}
Filip Cools.
\newblock {On the Singular Locus of Grassmann Secant Varieties}.
\newblock {\em Bulletin of the Belgian Mathematical Society-Simon Stevin},
  16(5):799--803, 2009.

\bibitem[CvdW37]{chow}
Wei-Liang Chow and Bartel~L. van~der Waerden.
\newblock {Zur algebraischen Geometrie IX. \"Uber zugeordnete Formen und
  algebraische Systeme von algebraischen Mannigfaltigkeiten}.
\newblock {\em Mathematische Annalen}, 113:692--704, 1937.

\bibitem[DS95]{chowIntro}
John Dalbec and Bernd Sturmfels.
\newblock {Introduction to Chow Forms}.
\newblock In {\em Invariant Methods in Discrete and Computational Geometry},
  pages 37--58. Springer, 1995.

\bibitem[GKZ94]{gkz}
Israel~M. Gel'fand, Mikhail~M. Kapranov, and Andrei~V. Zelevinsky.
\newblock {\em {Discriminants, Resultants and Multidimensional Determinants}},
  volume 227 of {\em Graduate Texts in Mathematics}.
\newblock Birkh\"auser, 1994.

\bibitem[GM86]{green_morr}
Mark~L. Green and Ian Morrison.
\newblock {The Equations Defining Chow Varieties}.
\newblock {\em Duke Mathematical Journal}, 53:733--747, 1986.

\bibitem[Har92]{harris}
Joe Harris.
\newblock {\em {Algebraic Geometry: A First Course}}, volume 133 of {\em
  Graduate Texts in Mathematics}.
\newblock Springer, 1992.

\bibitem[Hol88]{holme}
Audun Holme.
\newblock {The Geometric and Numerical Properties of Duality in Projective
  Algebraic Geometry}.
\newblock {\em manuscripta mathematica}, 61:145--162, 1988.

\bibitem[KNT17]{congruences}
Kathl\'en Kohn, Bernt Ivar~Utst\o{}l N\o{}dland, and Paolo Tripoli.
\newblock {Secants, Bitangents, and Their Congruences}.
\newblock In G.G. Smith and B.~Sturmfels, editors, {\em Combinatorial Algebraic
  Geometry}, volume~80 of {\em Fields Institute Communications}, pages 87--112.
  Springer, 2017.

\bibitem[Koh19]{coisotropicHS}
Kathl\'en Kohn.
\newblock {Coisotropic Hypersurfaces in Grassmannians}.
\newblock {\em Journal of Symbolic Computation (to appear, DOI:
  10.1016/j.jsc.2019.12.002)}, 2019.

\bibitem[Pie76]{pieneCurves}
Ragni Piene.
\newblock {Numerical Characters of a Curve in Projective $n$-Space}.
\newblock In P.~Holm, editor, {\em Real and Complex Singularities, Oslo}, pages
  475--495, 1976.

\bibitem[Pie78]{piene}
Ragni Piene.
\newblock {Polar Classes of Singular Varieties}.
\newblock {\em Annales Scientifiques de l’\'Ecole Normale Sup\'erieure},
  11:247--276, 1978.

\bibitem[Pie81]{pieneCuspProj}
Ragni Piene.
\newblock {Cuspidal Projections of Space Curves}.
\newblock {\em Mathematische Annalen}, 256:95--119, 1981.

\bibitem[Stu17]{hurwitz}
Bernd Sturmfels.
\newblock {The Hurwitz Form of a Projective Variety}.
\newblock {\em Journal of Symbolic Computation}, 79:186--196, 2017.

\bibitem[Wal56]{wallace}
Andrew~H. Wallace.
\newblock {Tangency and Duality over Arbitrary Fields}.
\newblock {\em Proceedings of the London Mathematical Society}, 3(3):321--342,
  1956.

\end{thebibliography}

\end{document}